\documentclass[11pt,leqno,textwidth=10cm]{amsart}
\usepackage{amssymb}
\allowdisplaybreaks
\usepackage{mathrsfs}
\usepackage{amsmath, amssymb, mathtools, nicefrac}
\usepackage{csquotes}
\usepackage{abstract}
\usepackage[hidelinks]{hyperref}
\usepackage{marginnote}
\usepackage{xcolor}

\usepackage{etoolbox}
\makeatletter
\patchcmd{\@maketitle}{\newpage}{}{}{} 
\makeatother
\numberwithin{equation}{section}

\theoremstyle{definition}
\newtheorem{definition}{Definition}[section]

\newtheorem{remark}[definition]{Remark}
\theoremstyle{plain}
\newtheorem{theorem}[definition]{Theorem}
\newtheorem{lemma}[definition]{Lemma}
\newtheorem{corollary}[definition]{Corollary}
\newtheorem{prop}[definition]{Proposition}

\usepackage{nicefrac}
\usepackage{csquotes}

\usepackage{etoolbox}
\makeatletter
\patchcmd{\@maketitle}{\newpage}{}{}{} 
\makeatother
\numberwithin{equation}{section}

\usepackage{amsopn}

\newcommand{\C}{\mathbb{C}}

\newcommand{\g}{{g}}

\newcommand{\M}{M}

\newcommand{\N}{\mathbb{N}}

\renewcommand{\O}[1]{\mathcal{O}\left(#1\right)}

\newcommand{\R}{\mathbb{R}}

\newcommand{\vol}[1]{{\text{\normalfont{vol}}}_{#1}}

\renewcommand{\epsilon}{\varepsilon}
\renewcommand{\phi}{\varphi}

\newcommand{\del}{\partial}
\newcommand{\nabbar}{{\nabla}}
\newcommand{\Lap}{\Delta}
\renewcommand{\div}{\text{\normalfont{div}}}

\newcommand{\numberthis}{\addtocounter{equation}{1}\tag{\theequation}}

\renewcommand{\theequation}{\arabic{section}.\arabic{equation}}

\newcommand{\equp}[1]{\mathrel{\mathop{=}^{\mathrm{#1}}}}
\newcommand{\lequp}[1]{\mathrel{\mathop{\leq}^{\mathrm{#1}}}}

\newcommand{\todo}[1]{#1}
\usepackage{soul}

\begin{document}


\title[]{Blow-up of waves on singular spacetimes with generic spatial metrics}
\author[D.~Fajman, L.~Urban]{David Fajman, Liam Urban}

\date{\today}


\address{
\begin{tabular}[h]{l@{\extracolsep{8em}}l} 
David Fajman,  & Liam Urban, \\
Faculty of Physics, & Faculty of Mathematics,\\ 
University of Vienna, & University of Vienna, \\
Boltzmanngasse 5, & Oskar-Morgenstern-Platz 1,  \\
1090 Vienna, Austria & 1090 Vienna, Austria\\
David.Fajman@ univie.ac.at, & liam.urban@ univie.ac.at 
\end{tabular}
}

\maketitle
\begin{abstract}
We study the asymptotic behaviour of solutions to the linear wave equation on cosmological spacetimes with Big Bang singularities and show that appropriately rescaled waves converge against a blow-up profile. Our class of spacetimes includes Friedman-Lema\^{\i}tre-Robertson-Walker (FLRW) spacetimes with negative sectional curvature that solve the Einstein equations in presence of a perfect irrotational fluid with $p=(\gamma-1)\rho$. As such, these results are closely related to the still open problem of past nonlinear stability of such FLRW spacetimes within the Einstein scalar field equations. In contrast to earlier works, our results hold for spatial metrics of arbitrary geometry, hence indicating that the matter blow-up in the aforementioned problem is not dependent on spatial geometry. Additionally, we use the energy estimates derived in the proof in order to formulate open conditions on the initial data that ensure a non-trivial blow-up profile, for initial data sufficiently close to the Big Bang singularity and \todo{with less harsh assumptions} for $\gamma<2$.
\end{abstract}


\maketitle

\section{Introduction}

This paper is concerned with the blow-up of waves on cosmological backgrounds toward the Big Bang singularity. More precisely, spacetimes of physical interest considered in the following are spatially flat or hyperbolic Friedman-Lema\^\i tre-Robertson-Walker (FLRW) spacetimes $\left(\overline{M},\overline{g}\right)$ which take the form $\overline{M}=\R^+\times\M$, where $(\M,\g)$ is a three-dimensional closed (connected) Riemannian manifold of vanishing or negative constant sectional curvature, endowed with the metric $\overline{g}=-dt^2+a(t)^2\g$
for some smooth scale factor $a$. These arise by assuming the universe to be spatially homogeneous and isotropic, both of which are in accordance with what is physically observed on large scales. To derive a reasonable model of our universe and the Big Bang in particular, it is necessary to consider such FLRW spacetimes that solve the Einstein equations in presence of matter. The most common approach, and one of the simplest, models the universe as an irrotational ideal fluid with energy density $\rho$ and pressure $p$, and assumes the linear equation of phase $p=(\gamma-1)\rho$ for $\gamma\in(\nicefrac23,2]$. For $\gamma-1\geq 0$, this can be interpreted as the square of the speed of sound $c_s$ within the fluid. The upper bound $\gamma=2$ then corresponds to a stiff fluid, i.e.~$c_s=c=1$, while spacetimes with $\gamma=\nicefrac23$ do not admit a past singularity. An FLRW spacetime with constant spatial sectional curvature $\kappa$ then solves the resulting Einstein-Euler-system if and only if the Friedman equations (see \eqref{eq:Friedmann-1} and \eqref{eq:Friedmann-2}) are satisfied.\\

FLRW spacetimes are covered by the famous Hawking singularity theorem (see \cite{Hawk67}) which states that a vast number of globally hyperbolic spacetimes are geodesically incomplete. More precisely, FLRW spacetimes exhibit a Big Bang in the sense that the Kretschmann scalar blows up as $t\downarrow 0$, thus making the spacetime (past) $C^2$-inextendible. The Strong Cosmic Censorship conjecture postulates that this is, in fact, generically the case in cosmological settings -- else, this would necessitate different inequivalent extensions and hence violate determinism. Thus, it is of vital importance to the validity of the FLRW model when applied to the observable universe that the Strong Cosmic Censorship conjecture holds for spacetimes that are \enquote{close} to FLRW spacetimes, i.e.~that their singularity formation is linearly stable, or even nonlinearly stable, within the respective Einstein equations.\\
For nonnegative sectional curvature, a full picture was obtained in \cite{Rodnianski2014,Rodnianski2018,Speck2018} for the Einstein scalar field and stiff fluid systems, with similar results even available for the scalar field system near subcritical Kasner spacetimes as shown in \cite{fournodavlos2020stable}. The goal of this paper is to provide a step toward the still open problem of nonlinear stability in $\kappa=-1$ in these matter models essentially by analysing the precise blow-up behaviour of the matter component within the Einstein scalar-field system -- namely waves -- on a fixed FLRW background in which the scale factor satisfies the Friedman equations given by an ideal fluid. Since the scalar field model is essentially a sub-case of the stiff Einstein-Euler-system (see \cite[p.~31]{Christ07} for more details), this choice of background, especially in the stiff fluid case, provides a scale factor that one can expect to behave similarly to that of the coupled system while still wholly decoupling the scalar field from the spacetime geometry itself.\\

In this paper, we observe that \textit{the precise Riemannian spatial geometry is irrelevant for our blow-up analysis beyond its influence on the scale factor via the Friedman equations}. In this spirit, we will consider \enquote{warped product spacetimes} of the form $\left(\overline{M},\overline{g}\right)=(\R^+\times\M,-dt^2+a(t)^2\g),$ where $(\M,\g)$ is now simply a closed three-dimensional Riemannian manifold without any assumptions on curvature. Further, we consider such spacetimes \enquote{of type $0$ (resp.~$-1$)} where the scale factor $a$ solves the Friedman and continuity equations (see \eqref{eq:Friedmann-1a} and \eqref{eq:a-cont-eq}) associated with $\kappa=0$ (resp. $\kappa=-1$) within the equation of state $p=(\gamma-1)\rho,\ \gamma\in(\nicefrac23,2]$. For type 0, this simply means $a(t)=t^{\nicefrac2{3\gamma}}$, while the scale factor in type $-1$ behaves like $t^{\nicefrac2{3\gamma}}$ asymptotically as $t\downarrow 0$. In short, we endow an inhomogeneous spatial manifold with the scale factor expected for flat or negative constant sectional curvature \textit{in presence of fluid matter}. This result indicates that inhomogeneities within the spatial geometry have no impact on the blow-up of scalar field matter (or, more generally, that of scalar fields on fluid backgrounds).

To summarize, we will analyze the blow-up of solutions $\psi:\overline{M}\rightarrow\R$ to the wave equation
\[\square_{\overline{g}}\psi=-\del_t^2\psi(t,\cdot)+a(t)^{-2}(\Lap\psi)(t,\cdot)-3\frac{\dot{a}(t)}{a(t)}(\del_t\psi)(t,\cdot)=0\quad \forall\,t>0\,\]
on warped product spacetimes of type $0$ and $-1$, where $\Lap=g^{ij}\nabla_i\nabla_j$ is the Laplace-Beltrami operator with respect to $(\M,g)$. Our main result is:
\begin{theorem}\label{thm:main}
Let
\[\left(\overline{M}=\R^+\times\M,\overline{g}=-dt^2+a(t)^2\g\right)\]
be a warped product spacetime of type $0$ or $-1$, where the scale factor $a$ solves the Friedman and continuity equations, see \eqref{eq:a-cont-eq}, \eqref{eq:Friedmann-1}, \eqref{eq:Friedmann-2}, associated with the equation of state $p=(\gamma-1)\rho$ for $\gamma\in(\frac23,2]$. Let $\psi$ be a smooth solution to the wave equation $\square_{\overline{g}}\psi=0$ on $\left(\overline{M},\overline{g}\right)$.\\
Further, choose a fixed non-zero spatially homogeneous solution $\psi_{\text{hom}}$ to the wave equation -- more precisely, by Remark \ref{rem:stationary-waves}, we choose
\[\psi_{\text{hom}}(t)=\begin{cases}
t^{1-\frac2{\gamma}} & \text{type } $0$, \gamma<2\\
\log(t) & \text{type } $0$, \gamma=2\\
\int_t^\infty a(s)^{-3}\,ds & \text{type } $-1$
\end{cases}\,.\]
Then, there exist unique functions $A\in C^\infty(\M), r\in C^\infty(\overline{M})$ such that
\[\psi(t,x)=A(x)\psi_{\text{hom}}(t)+r(t,x) \]
and where \[\frac{r(t,x)}{\psi_{\text{hom}}(t)}\to 0\text{ as }t\downarrow 0\] uniformly in $x\in\M$.\\
If $\gamma<2$ holds, then ${\psi(t,\cdot)}{\left(\psi_{\text{hom}}(t)\right)^{-1}}$ converges to $A$ in $C^\infty(\M)$.
\end{theorem}

In type $-1$ warped products, these homogeneous waves exhibit precisely the analogous asymptotic behaviour toward $t\to 0$ as in type $0$. Thus, our main theorem already gives the precise highest order blow-up for waves on both types of backgrounds (type $-1$ being the more central), along with a very strong control on the error terms for the non-stiff setting.\\

\todo{As shown in \cite{Girao19}, it does not even hold in the FLRW setting that any wave exhibits blow-up toward the Big Bang hypersurface.} To furthermore show that the blow-up of highest possible order is actually generic, we will establish \todo{an open condition} on the initial data $(\psi(t_0,\cdot),\del_t\psi(t_0,\cdot))$ on a hypersurface $\M_{t_0}=\{t_0\}\times\M$ for \todo{ small enough  $t_0>0$} such that $A$ \todo{doesn't become zero pointwise, see Theorem \ref{thm:pointwise-blowup-stiff-incl} and Corollary \ref{cor:pointwise-blowup-stiff-incl}. This condition is applicable for any $\gamma\in(\nicefrac23,2]$, but has the downside of requiring initial data control in $H^5$ and requires a number of abstract constants depending on $t_0$, $a$ and the spatial geometry to state it. Restricting ourselves to $\gamma<2$ allows us to formulate significantly simpler conditions that ensure $A$} isn't identically zero (see Theorems \ref{thm:Global-Blowup} and \ref{thm:Global-Blowup-hyp}) or never vanishes (see Theorem \ref{thm:pointwise-blowup}) \todo{by showing that $\psi$ converges to $A\psi_{\text{hom}}$ in an energy sense (see Proposition \ref{thm:energy-convergence})}. In essence and brushing over some of the technical details for now, \todo{the conditions on global blow-up} require that the initial data must be \textit{velocity term dominated} in the sense that the $L^2$ norm of $\del_t\psi(t_0,\cdot)$ must be sufficiently large compared to $L^2$-norms of spatial derivatives of $\psi(t_0,\cdot)$ and $\del_t\psi(t_0,\cdot)$ of up to third order, while \todo{all pointwise conditions require} the initial data to be close to that of an (a priori specified) homogeneous wave in \todo{an approriate Sobolev sense, which can be improved from $H^5$ to $H^3$ in non-stiff regimes}. Thus, \todo{all of the listed conditions} state that waves with almost homogeneous data remain almost homogeneous and specify the degree to which blow-up of homogeneous waves is stable under perturbation of initial data. While \todo{not all of these statements extend} to stiff fluid backgrounds, \todo{even those that don't} still provide a good heuristic for similar arguments in the full scalar field system and as to \todo{what the precise behaviour of matter will be and in what ways some care needs to be taken when coupling matter and geometry.}\\

The behaviour of solutions to linear wave and Klein Gordon equations has been studied extensively, with many recent results as in \cite{Allen09,Franzen18, Bach19,Ring19, Ring21wave, Ring21linear}. In particular, compared to the work done by A.~Alho, G.~Fournodavlos and A.~T.~Franzen in \cite{Franzen18}, where this problem was considered on flat FLRW spacetimes (and Kasner spacetimes) specifically with similar conclusions, we provide a significant extension in multiple ways: Not only do we explicitly include the stiff case $\gamma=2$ in our analysis as much as possible and extend the results to hyperbolic FLRW spacetimes, but even to the mathematically broader class of warped product spacetimes. To this end, we similarly consider energies adapted to the structure of our spacetimes, obtain energy estimates that can be improved to pointwise estimates on waves $\psi$ as well as on waves rescaled by the suspected leading order $\psi_{\text{hom}}$. The key difference to \cite{Franzen18} is that, to move from energy to pointwise estimates, we can no longer just commute the wave operator with arbitrary spatial coordinate derivatives as is possible in the spatially flat setting of \cite{Franzen18}. Instead, we solely rely on the fact that the spatial Laplace-Beltrami operator $\Lap$ is elliptic and commutes with the wave operator for any such step. As such, we don't just extend the results of \cite{Franzen18}, but actually uncover that the asymptotic behaviour of waves, even in the \enquote{special case} of FLRW spacetimes with flat spatial geometry, is driven solely by the scale factor and not in any way influenced by the geometry itself. Note that, while our results turn out to be contained in the more general analysis of {\cite{Ring21linear}} upon close inspection, the Fourier approach used to attain them in {\cite{Ring21linear}} is at least difficult to extend to the full nonlinear problem, where we believe our direct approach may turn out to be more useful. Further, since we are concerned with the isotropic setting, it isn't clear whether the methods contained in {\cite{Ring21wave}} could be extended to our setting: There, anisotropy is crucial in the sense that the second fundamental form must have distinct eigenvalues.\\
Furthermore, we point to \cite{beyer2020relativistic}, where general solutions to the relativistic Euler equations were constructed on Kasner backgrounds that solved the vacuum and Einstein scalar-field equations, thus forming somewhat of a counter-part to our analysis.

Since our approach even circumvents any choice of local or global frame, it also indicates how the strategies in \cite{Rodnianski2014,Rodnianski2018,Speck2018} could be adapted to yield nonlinear stability of the respective Einstein equations in hyperbolic spatial geometry: All rely on using global frames to reach higher order energy estimates and thus sufficiently control the solution variables, which again may not prove to be well adapted to considering the full Einstein system over the various closed manifolds with $\kappa=-1$. Within the correct gauge, we hence suspect that suitable elliptic differential operators which (almost) commute with the evolutionary system should be completely sufficient for this task and thus help extend the results to the hyperbolic case. That our open blow-up conditions hold also suggests this in the sense that sufficiently small perturbations of FLRW initial data will preserve the asymptotic behaviour of the FLRW model.\\

\textbf{Acknowledgements.} L.U. acknowledges the support by the START-Project Y963-N35 of the Austrian Science Fund (FWF). L.U. also thanks the \enquote{Studienstiftung des Deutschen Volkes} for their scholarship. 

\section{Preliminaries}\label{sec:setup}

\subsection{Setting}\label{subsec:setting}

As outlined in the introduction, our starting point for the choice of FLRW background
\begin{equation*}
\left(\overline{M}=I\times\M,\overline{g}=-dt^2+a(t)^2\g_{\kappa}\right)
\end{equation*}
is to consider solutions to the perfect fluid model with equation of state $p=(\gamma-1)\rho$ for $\gamma\in(\nicefrac23,2]$. It is a standard result \cite[p.~345f.
]{ONeill83} that, given constant sectional curvature $\kappa$ on the spatial manifold $(\M,\g)$, the \textit{continuity equation}
\begin{equation}\label{eq:a-cont-eq}
\del_t\rho=-3\frac{\dot{a}}{a}(\rho+p)=-3\frac{\dot{a}}{a}\gamma\rho
\end{equation}
and the \textit{Friedman equations}
\begin{align}
\left(\frac{\dot{a}}{a}\right)^2&=\frac{8\pi}3\rho-\frac{\kappa}{a^2} \label{eq:Friedmann-1}\\
\frac{\ddot{a}}{a}&=-\frac{4\pi}3\left(\rho+3p\right)\, \label{eq:Friedmann-2}
\end{align}
are satisfied, where \eqref{eq:a-cont-eq} and \eqref{eq:Friedmann-1} imply \eqref{eq:Friedmann-2}.
Since \eqref{eq:a-cont-eq} is uniquely solved by 
\begin{equation}\label{eq:density-expansion-relation}
\rho(t)=B\cdot a(t)^{-3\gamma}\,
\end{equation}
for $B\in\R$, it is equivalent to require the scale factor to satisfy
\begin{equation}\label{eq:Friedmann-1a}
\dot{a}=\sqrt{\frac{8\pi B}3a^{2-3\gamma}-\kappa}\ .
\end{equation}
We will further amend this classic cosmological framework as follows:

\begin{itemize}
\item We will call a spacetime $\left(\overline{M},\overline{g}\right)$ a \textbf{warped product spacetime} (or simply warped product) if it satisfies all conditions of an FLRW spacetime except that the spatial manifold need not have constant sectional curvature, i.e.~if 
\begin{equation}\label{eq:wp}
(\overline{M}=I\times\M,\overline{g}=-dt^2+a(t)^2\g)
\end{equation}
for an open interval $I$, a three-dimensional Riemannian manifold $\left(\M,\g\right)$ and $a\in C^\infty(I,\R^+)$. In the following, $\left(\overline{M},\overline{g}\right)$ will always denote a warped product unless stated otherwise.
\item We say a warped product spacetime has a \textbf{Big Bang singularity at $t_{min}$} when $a\to 0$ and $\dot{a}\to \infty$ hold approaching $t_{min}$ (see \cite[p.348, Def. 12.16]{ONeill83}). It is additionally called \enquote{physical} if $\rho\to\infty$ holds toward $t_{min}$.
\item We will restrict ourselves to considering scale factors associated with spatial geometries with nonpositive constant sectional curvature. In particular, we call $\left(\overline{M},\overline{g}\right)$ of \textbf{type 0} (resp. \textbf{type --1}) if $a$ satisfies \eqref{eq:Friedmann-1a} for $\kappa=0$ (resp. $\kappa=-1$). Concretely, $a$ takes the form of \eqref{eq:scale-factor-0} or simply $t^{\frac2{3\gamma}}$ in type $0$, while the behaviour in type $-1$ is discussed in Lemma \ref{lem:scale-factor}. In other words, we assume the scale factor to take the form it takes in the true FLRW setting with the ideal fluid model for our background, and then relax the assumptions on the spatial geometry.\\
Further, as we will later show, we can and thus will assume $I=\R^+$ w.l.o.g. for the domain of our scale factors, with $a(0)=0$.
\item We will always assume $\rho>0$ (i.e.~$B>0$ in \eqref{eq:density-expansion-relation}). However, considering our equation of state $p=(\gamma-1)\rho$ and $\nicefrac23<\gamma\leq2$, we allow for what would be negative pressure in the FLRW model. While the main physical interpretation of this equation of state arises for $\gamma\geq1$ (here, $\gamma-1=c_s^2$, where $c_s$ is the speed of sound within the fluid), we extend to $\gamma>\nicefrac23$ because this allows us to consider all choices of $\gamma$ where a Big Bang singularity actually occurs mathematically. 
Finally, it should be noted that a dust filled universe is associated to $\gamma=1$, while a radiation filled universe corresponds to $\gamma=\nicefrac43$ (see \cite[Chapters 6.4.5, 6.4.6]{Chrusciel10}).
\item $\psi$ will always denote a smooth wave on a spacetime $(M,g)$.
\end{itemize}

\subsection{Analysis of the scale factor}\label{subsec:scale-factor}

Getting a clear grasp on our choice of scale factor is very simple in type $0$ warped product spacetimes: Here, \eqref{eq:Friedmann-1a} with initial condition $a(0)=0$ has the unique solution
\begin{equation}\label{eq:scale-factor-0} a(t)=\left(\frac{3\gamma}2\sqrt{\frac{8\pi B}3}\right)^{\frac2{3\gamma}}t^{\frac2{3\gamma}}\,.\end{equation}
for all $t>0$, hence we can assume
\begin{equation}\label{eq:scale-factor-0-wlog}
a(t)=t^{\frac2{3\gamma}}
\end{equation}
for $\nicefrac23<\gamma\leq 2$ by scaling in all following calculations. In type $-1$, the situation is roughly similar, but the analysis is more involved:

\begin{lemma}\label{lem:scale-factor}
The initial value problem
\begin{equation}\label{eq:ivp}
\dot{a}=f(a):=\sqrt{\frac{8\pi B}3a^{2-3\gamma}+1},\quad a(0)=0
\end{equation}
has a unique solution $a:[0,\infty)\rightarrow\R, B>0, \gamma\in(\nicefrac23,2]$ with the following properties:
\begin{enumerate}
\item $a$ is strictly increasing.
\item $a(t)\geq t\text{ for all }t\geq0$, with equality only at $t=0$.
\item $a\in C([0,\infty),[0,\infty))\cap C^\omega(0,\infty),(0,\infty))$
\item ${a(t)}\simeq{t^{\frac2{3\gamma}}}\text{ as }t\to 0$
\item $\int_t^{\infty}a(s)^{-3}ds<\infty$ for all $t>0$
\item For $t_0>0$ small enough, $0<t<t_0$, one has
\end{enumerate}
\[\int_t^{t_0} a(s)^{-3}ds\simeq\begin{cases}
{t^{1-\frac2{\gamma}}}-t_0^{1-\frac2\gamma}& \gamma<2\\
\log(t_0)-\log(t)& \gamma=2
\end{cases}\,.\]
\end{lemma}
\begin{proof}
The first two points are immediate once recognizing that any solution with this initial condition must immediately be nonnegative.\footnote{Else, even if $a^{2-3\gamma}$ is well-defined, $a$ must be bounded away from $0$ for the square root to be well defined, so the initial condition can't be satisfied.} The fifth point also immediately follows from the second, once we have shown $a$ to be defined on $(0,\infty)$.\\
We now move to the shifted initial value problem
\begin{equation}\label{eq:shifted-ivp}
\dot{a}=f(a),\quad a(t_0)=a_0>0 
\end{equation}
at $t_0>0$. Since $f:(0,\infty)\longrightarrow\R$ is smooth, a unique smooth real-valued solution exists on some maximal interval of existence $I=(t_{min},t_{max})$, and as $f$ is monotonously decreasing, one has
\[a(t)\leq \sqrt{1+\frac{8\pi B}3a_0^{2-3\gamma}}\cdot(t-t_0)+a_0\,.\]
In particular, it follows that $t_{max}=\infty$. 

Furthermore, because $a$ is strictly increasing and positive on $I$, $a$ converges approaching $t_{min}$, and due to maximality, it follows that $a(t)\rightarrow0$ as $t\downarrow t_{min}$. Finally, $t\in(0,\infty)\mapsto a(t+t_{min})$ now solves \eqref{eq:ivp} -- or equivalently, we can assume $t_{min}=0$ without loss of generality for a solution of \eqref{eq:shifted-ivp} since no solution can be extended past $0$.\\

Additionally, $f$ extends to a holomorphic function on the simply connected set\linebreak $\C\backslash\{z\in\C,\,\text{Im}(z)\geq 0\}$ by appropriate choice of logarithm, thus \eqref{eq:shifted-ivp} has a unique holomorphic local solution around any $t_0\in I$ with initial condition $a(t_0)\in\R$ by the Cauchy-Kovalevskaya Theorem \cite[p.46f.]{Folland95}. Hence, this local uniqueness yields a real analytic solution on $I=(0,\infty)$ to the real-valued differential equation that must agree with any real solution on $I$, so any real solution on $I$ must be analytic.\\

On the other hand, assume there were two different (maximally extended) solutions $a_1,a_2$ to \eqref{eq:ivp}, then some $\tilde{a}>0$ has to exist such that $a_1(t_1)=\tilde{a}=a_2(t_2)$ for some $0<t_1<t_2$. However, both $a_1$ and $t\mapsto a_2(t+t_2-t_1)$ locally solve the initial value problem
\[\dot{\phi}(t)=f(a),\ \phi(t_1)=\tilde{a}\,.\]
Its solutions are locally unique and (as argued before) analytic on their open existence intervals, hence any two local solutions are extendible to a common maximal solution. In par\-ti\-cu\-lar, it would follow that $a_2(t_2-t_1)=a_1(0)=0$. Since $a_2$ is strictly increasing, $t_2-t_1=0$ would have to hold, which is a contradiction. Hence, \eqref{eq:ivp} has a unique continuous solution on $[0,\infty)$ which must then also be analytic on $(0,\infty)$.\\

To prove the asymptotic behaviour of $a$, consider $b(t):=a(t)^\frac{3\gamma}2$ which satisfies
\begin{equation*}
\dot{b}=\frac{3\gamma}2a^{\frac{3\gamma}2-1}\dot{a}=\frac{3\gamma}2\sqrt{a^{3\gamma-2}+\frac{8\pi B}{3}}.\\
\end{equation*}
We obtain $\displaystyle\lim_{t\to 0}\dot{b}(t)=\frac{3\gamma}2\sqrt{\frac{8\pi B}3}> 0$. By the l'Hospital rule, it now follows that
\[\lim_{t\to 0}\frac{a(t)}{t^{\frac2{3\gamma}}}=\left(\lim_{t\to 0}\frac{b(t)}{t}\right)^{\frac2{3\gamma}}=\lim_{t\to 0}\left(\dot{b}(t)\right)^{\frac2{3\gamma}}>0\,,\]
which shows the fourth point, immediately yielding the final one as well.
\end{proof}

\subsection{The wave operator and homogeneous waves}\label{subsec:basics}

Before moving on to the energy estimates fundamental to our results, we quickly derive some basic properties of the wave operator. By simply writing out the Christoffel symbols involved, one even sees on warped product spacetimes (see \eqref{eq:wp}) that the wave operator takes the form
\[\square_{\overline{g}}\phi(t,\cdot)=-\del_t^2\phi(t,\cdot)+a(t)^{-2}\Lap\phi(t,\cdot)-3\frac{\dot{a}(t)}{a(t)}\del_t\phi(t,\cdot) \quad \forall\,t>0,\phi\in C^\infty\left(\overline{M}\right),\]
where $\Lap\equiv\Lap_{\g}=\g^{ij}\nabla_i\nabla_j$ is the Laplace operator on $\M$. Thus:

\begin{corollary}\label{cor:wave-formulas}
For any smooth wave $\psi$ and any $t>0$, it holds that
\[\left(\del_t^2\psi\right)(t,\cdot)=a(t)^{-2}\Lap\psi(t,\cdot)-3\frac{\dot{a}(t)}{a(t)}(\del_t\psi)(t,\cdot).\]
Furthermore, for any $N\in\N_0$, $\Lap^N\psi: (t,x)\mapsto\left(\Lap^N\psi(t,\cdot)\right)(x)$ is also a smooth wave.
\end{corollary}
Note that the latter statement, along with the fact that $\Lap$ is elliptic, will be central to yielding higher order energy estimates and with it sufficiently strong control on $\psi$.

\begin{remark}\label{rem:stationary-waves}

Homogeneous waves are thus given by the differential equation
\[\del_t(a^3\del_t\psi)(t)=0\quad \forall t>0\]
after rearranging. In type $0$, they hence take the explicit form
\begin{equation}\label{eq:hom-waves-flat}
\psi(t)=\begin{cases}
C_1t^{1-\frac2{\gamma}}+C_2 & \gamma\in(\nicefrac23,2)\\
C_1\log(t)+C_2& \gamma=2
\end{cases}
\end{equation}
while one can use the fifth point in Lemma \ref{lem:scale-factor} for type $-1$ to write the homogeneous waves as
\begin{equation}
\psi(t)=C_1\int_t^\infty a(s)^{-3}\,ds+C_2\ \text{for}\ C_1,C_2\in\R\,.
\end{equation}
In particular, in either setting, \textbf{we thus expect waves to behave like $t^{1-\nicefrac2{\gamma}}$ towards the Big Bang singularity in warped product spacetimes that don't arise from stiff fluids, and like $\log(t)$ in the stiff case} (even for type $-1$, see the final point in Lemma \ref{lem:scale-factor}). In the following, when referring to homogeneous waves, it will always be assumed that they vanish in the far field where this is possible ($C_2=0$) and are not constant ($C_1\neq0$).
\end{remark}

\section{Energy estimates}\label{sec:energy-estimates}

For a smooth function $\phi:\overline{M}\rightarrow\R$, consider the following energies:
\begin{align}
E(t,\phi)&=\, E(\phi(t,\cdot))=\int_{\M}\left\lvert\del_t\phi(t,\cdot)\right\rvert^2+a(t)^{-2}\left\lvert{\nabbar}\phi(t,\cdot)\right\rvert_{\g}^2\,\vol{\M} \label{eq:energy-def-1}\\
E_N(t,\phi)&=E\left({\Lap}^N\phi(t,\cdot)\right)\label{eq:energy-def-2}
\end{align}

\subsection{Wave energy estimates}\label{subsec:energy-estimates}

For homogeneous waves, the energy of order $N=0$ is easily seen to take form
\[E(t,\psi_{\text{hom}})=\left\lvert Ca(t)^{-3}\right\rvert^2=C^2a(t)^{-6},\]
and in type $0$ specifically
\[E(t,\psi_{\text{hom}})= C^2t^{-\frac4{\gamma}}\]
for $C\in\R, C\neq 0$, by \eqref{eq:scale-factor-0-wlog}. The next proposition thus extends this observation to all waves:

\begin{prop}\label{prop:energy-1}
For any $N\in\N$ and $0<t<t_0$, the following estimate holds on any warped product spacetime $\left(\overline{M},\overline{g}\right)$ of type $0$ or $-1$:
\begin{equation*}
a(t)^{6}E_N(t,\psi)\leq a(t_0)^6 E_N(t_0,\psi)\,.
\end{equation*}
\end{prop}

Conceptually, we can stick fairly close to the proof of Propositon 2.1 in \cite{Franzen18} for this estimate while being cautious that everything can be generalized to warped product spacetimes, but we repeat the argument here for the sake of completeness.
\begin{definition}\label{def:energy-flux}
The \textbf{energy flux} $J^X[\phi]$ is the covector field defined by the projection of the energy-momentum tensor of scalar field matter along the vector field $X\in\mathcal{X}(M)$, i.e.~one defines
\[J_{\mu}^X[\phi]=X^{\nu}T_{\mu\nu}[\phi]=X^{\nu}\left(\overline{\nabla}_{\mu}\phi\overline{\nabla}_{\nu}\phi-\frac12\overline{g}_{\mu\nu}\overline{\nabla}^{\sigma}\phi\overline{\nabla}_{\sigma}\phi\right)\,.\]
\end{definition}
for a smooth function $\phi:\overline{M}\rightarrow\R$. Note:
\[J_0^{\del_t}[\phi]=T_{00}[\phi]=\frac12 \left(\lvert\del_t\phi\rvert^2+a(t)^{-2}\left\lvert\nabbar\phi\right\rvert_{\g}^2\right)\]
\begin{proof}[Proof of Proposition \ref{prop:energy-1}]
Set $X=a(t)^3\del_t$. Then, one computes:
\begin{align*}
\overline{\nabla}^\mu X^\nu=\,&{\overline{g}}^{\mu\sigma}\left[\left(\del_\sigma a^3\right)\del_t^\nu+a^3\overline{\nabla}_\sigma\del_t^\nu\right]\\
=&\begin{cases}
-3a^2\dot{a}& \mu=\nu=0\\
\overline{g}^{\mu\sigma}a^3\bar{\Gamma}_{0\sigma}^\nu=\overline{g}^{\mu\nu}\dot{a}a^2=\g^{\mu\nu}\dot{a} & \mu,\nu\neq 0\\
0&\text{else}
\end{cases}
\end{align*}
Thus, recalling that, since $\psi$ is a wave, the divergence of $T$ vanishes, one calculates
\begin{align*}
\overline{\nabla}^\mu\left(J_\mu^X[\psi]\right)=\,&\left(\overline{\nabla}^\mu X^\nu\right)T_{\mu\nu}[\psi]
\\
=\,&-3a^2\dot{a}T_{00}[\psi]+\dot{a}\g^{ij}T_{ij}[\psi]\\
=\,&-2\dot{a}\left\lvert\nabbar\psi\right\rvert_{\g}^2\leq0 \numberthis\label{eq:divergence-energy-flux}
\end{align*}
since $\dot{a}>0$ by \eqref{eq:Friedmann-1a}. The induced volume form $\vol{\M_s}$ on $\M_s=\{s\}\times\M$ is given by $\vol{\M_s}=a(s)^3\vol{\M}$ by the Jacobi transformation law. Now, we choose the orientation on $\overline{M}$ such that $(-\del_t,\mathcal{B})$ is positively oriented for any positively oriented local basis $\mathcal{B}$ on $T\M$. Integrating over the volume form $\vol{\overline{M}}$ associated with said orientation, the divergence theorem yields
\begin{align*}
-\int_t^{t_0}\int_{\M_s}\div\left(J^X[\psi]\right)\vol{\M_s}ds=&\,\int_{[t,t_0]\times\M}\div\left(J^X[\psi]\right)\vol{\overline{M}}\\
=\,&\int_{\M_{t_0}}J_0^X[\psi]\vol{\M_{t_0}}-\int_{\M_{t}} J_0^X[\psi]\vol{\M_t}\\
=\,&\frac12a(t_0)^6E(t_0,\psi)-\frac12a(t)^{6}E(t,\psi),
\end{align*}
which can be rearranged to
\begin{equation}\label{eq:key-divergence-theorem}
a(t)^6E(t,\psi)=a(t_0)^{6}E(t_0,\psi)+2\int_{t}^{t_0}\int_{\M_s}\div\left(J^X[\psi]\right)\vol{\M_s}ds\,.
\end{equation}
Since the divergence term is nonpositive by \eqref{eq:divergence-energy-flux}, the statement now follows.
\end{proof}

\begin{corollary}\label{cor:pointwise-estimate}
In the setting of Proposition \ref{prop:energy-1}, with $(t,x)\in\overline{M},\,0<t<t_0$, the following estimate holds for any smooth wave $\psi$:
\begin{align*}
\left\lvert\Lap^N\psi(t,x)\right\rvert&\,\leq Ca(t_0)^3\left(\int_t^{t_0}a(s)^{-3}\,ds\right)\left(\sqrt{E_N(t_0,\psi)}+\sqrt{E_{N+1}(t_0,\psi)}\right)\\
&\quad +\left\lvert\Lap^N\psi(t_0,x)\right\rvert\,\numberthis
\end{align*}
where $C>0$ is a $\g$-dependent constant. In particular, in type $0$ warped products associated with $\gamma\in(\nicefrac23,2]$, it follows that
\begin{align*}
\left\lvert\Lap^N\psi(t,x)\right\rvert&\,\leq C{t_0}^{\frac2{\gamma}}\left(\sqrt{E_N(t_0,\psi)}+\sqrt{E_{N+1}(t_0,\psi)}\right)\begin{cases}
\frac{t^{1-\frac2{\gamma}}-{t_0}^{1-\frac{2}{\gamma}}}{\frac2{\gamma}-1} & \frac23<\gamma<2\\
\log(t_0)-\log(t) & \gamma=2\,
\end{cases}\\
&\quad +\left\lvert\Lap^N\psi(t_0,x)\right\rvert\numberthis\label{eq:pointwise-flat-precise}
\end{align*}
and this extends to warped products of type $-1$, choosing small enough $t_0>0$ and up\-da\-ting $C\equiv C(\g,t_0,\rho(t_0))$.
\end{corollary}
\begin{proof}
Applying in $(\ast)$ both a standard $L^2$ estimate for elliptic operators of second order (see \cite[p. 463, Theorem 27]{Besse08} and that $\Lap$ is elliptic for any Riemannian metric $\g$ (see \cite[p. 462, Example 19]{Besse08}, one computes
\begin{align*}
\Big\lvert\Lap^N\psi&(t,\cdot)\Big\rvert\leq\left\lvert\int_t^{t_0}\del_t\Lap^N\psi(s,x)ds\,\right\rvert+\left\lvert\,\Lap^N\psi(t_0,x)\right\rvert\\
&\leq\,\int_t^{t_0}\left\|\del_t\Lap^N\psi(s,\cdot)\right\|_{L^{\infty}\left(\M\right)}ds\,+\,\left\lvert\Lap^N\psi(t_0,x)\right\rvert\\
&\lequp{(\ast)}\,C\cdot\int_t^{t_0} \left(\left\|\del_t\Lap^N\psi(s,\cdot)\right\|_{L^2\left(\M\right)}+\left\|\del_t\Lap^{N+1}\psi(s,\cdot)\right\|_{L^2\left(\M\right)}\right)ds\,+\,\left\lvert\Lap^N\psi(t_0,x)\right\rvert\\
&\leq\,C\cdot\int_t^{t_0} \left(\sqrt{E_N(s,\psi)}+\sqrt{E_{N+1}(s,\psi)}\right)ds\,+\,\left\lvert\Lap^N\psi(t_0,x)\right\rvert\\
&\lequp{(\ast\ast)}\,C\cdot\left(\sqrt{E_N(t_0,\psi)}+\sqrt{E_{N+1}(t_0,\psi)}\right)\int_t^{t_0}\frac{a(t_0)^3}{a(s)^3}ds\,+\,\left\lvert\Lap^N\psi(t_0,x)\right\rvert\,,
\end{align*}
where $(\ast\ast)$ follows from Proposition \ref{prop:energy-1}.\\
In type $0$, \eqref{eq:pointwise-flat-precise} is simply obtained by computing the integral. Moving on to type $-1$, by the last point in Lemma \ref{lem:scale-factor}, one has
\[\int_t^{t_0}a(s)^{-3}\,ds\leq C_{t_0,\rho(t_0)} \int_t^{t_0}\left(s^{-\frac2{3\gamma}}\right)^3\,ds=\int_t^{t_0}s^{\frac2\gamma}\,ds\]
for $t_0>0$ small enough, and thus the final claim follows. 
\end{proof}

\subsection{Rescaled energy estimates}\label{subsec:rescaled-en-est}

To derive a more precise asymptotic behaviour, it is now intuitive to consider the analogous energies for waves rescaled by the leading order suggested by Proposition \ref{prop:energy-1} and Corollary \ref{cor:pointwise-estimate}. We start with type $0$ warped products:
\begin{prop}\label{prop:energy-2}
Let $\nicefrac23<\gamma<2$ and set
\[\beta=\max\left(\frac4{3\gamma},4-\frac4{\gamma}\right)\,.\]
For a smooth wave $\psi$ in a warped product spacetime $\left(\overline{M},\overline{g}\right)$ of type $0$, we set\linebreak $\hat{\psi}(t,x)=\nicefrac{\psi(t,x)}{t^{1-\nicefrac2{\gamma}}}$. Then, for any $N\in\N$ and $0<t<t_0$, the following estimates hold for a $\g$-dependent constant $C>0$:
\begin{equation*}
t^\beta E_N\left(t,\hat{\psi}\right)\leq t_0^{\beta} E_N\left(t_0,\hat{\psi}\right)\,,\\
\end{equation*}
\vspace{-0.7cm}
\begin{align*}
\left\lvert\Lap^N\hat{\psi}(t,\cdot)\right\rvert&\leq \frac{Ct_0^{\frac{\beta}2}}{1-\frac{\beta}2}\left(t_0^{1-\frac{\beta}2}-t^{1-\frac{\beta}2}\right)\left(\sqrt{E_N\left(t_0,\hat{\psi}\right)}+\sqrt{E_{N+1}\left(t_0,\hat{\psi}\right)}\right)\\
&\quad +\left\lvert\Lap^N\hat{\psi}(t_0,\cdot)\right\rvert \numberthis\label{eq:pointwise-rescaled-flat}
\end{align*}
\end{prop}
\begin{proof}
Again, it suffices to just prove the case $N=0$. First, one computes
\begin{align*}
\square_{\overline{g}}\hat{\psi}
=&-\frac2t\left(\frac2{\gamma}-1\right)\del_t\hat{\psi}\,.
\end{align*}
Now, one calculates: 
\begin{align*}
\del_t E\left(t,\hat{\psi}\right)=&\int_{\M}\left[2\del_t^2\hat{\psi}\cdot\del_t\hat{\psi}+2t^{-\frac{4}{3\gamma}}\cdot\g\left(\del_t\nabbar\hat{\psi},\nabbar\hat{\psi}\right)-\frac{4}{3\gamma}t^{-\frac{4}{3\gamma}-1}\left\lvert\nabbar\hat{\psi}\right\rvert_{\g}^2\right]\,\vol{\M}\\
=&\int_{\M}\left[2\left(t^{-\frac{4}{3\gamma}}\Lap\hat{\psi}+\frac{2}{t}\left(\frac1{\gamma}-1\right)\del_t\hat{\psi}\right)\del_t\hat{\psi}\right.\\
&\quad \left.-2t^{-\frac4{3\gamma}}\del_t\hat{\psi}\cdot\Lap\hat{\psi}-\frac{4}{3\gamma t}t^{-\frac{4}{3\gamma}}\left\lvert\nabbar\hat{\psi}\right\rvert_{\g}^2\right]\,\vol{\M}\\
=&\int_{\M}\left[\frac1t\left(\frac4{\gamma}-4\right)\left\lvert\del_t\hat{\psi}\right\rvert^2-\frac4{3\gamma t}t^{-\frac4{3\gamma}}\left\lvert\nabbar\hat{\psi}\right\rvert_{\g}^2\right]\vol{\M}\\
\geq& -\frac1t\max\left(4-\frac4{\gamma},\frac4{3\gamma}\right)\int_{\M}\left[\left\lvert\del_t\hat{\psi}\right\rvert^2+t^{-\frac4{3\gamma}}\left\lvert\nabbar\hat{\psi}(t,\cdot)\right\rvert_{\g}^2\right]\vol{\M}\\
=& -\frac\beta{t}E\left(t,\hat{\psi}\right)
\end{align*}
From here, we can deduce the first estimate with the Gronwall lemma. The pointwise estimate also follows analogously to Corollary \ref{cor:pointwise-estimate}, with
\begin{equation}\label{eq:leading-term-estimate}
\left\|\del_t\hat{\psi}(t,\cdot)\right\|_{L^\infty\left(\M\right)}\leq C\left(\frac{t_0}{t}\right)^{\frac{\beta}{2}}\left(\sqrt{E\left(t_0,\hat{\psi}\right)}+\sqrt{E_{1}\left(t_0,\hat{\psi}\right)}\right)
\end{equation}
for any $0<t<t_0,\,x\in\M$ (and similarly for $N>0$).
\end{proof}

\begin{remark}\label{rem:energy-2-consequences}
Note that one has
\[0<1-\frac{\beta}2=\begin{cases}
1-\frac2{3\gamma}& \frac2{3}<\gamma\leq\frac4{3}\\
\frac{2}{\gamma}-1 & \frac4{3}\leq\gamma<2
\end{cases},
\]
so the proof of Proposition \ref{prop:energy-2} also demonstrates that
\[t\mapsto\frac{\Lap^N\psi(t,x)}{t^{1-\frac2{\gamma}}}\] is absolutely continuous\footnote{Here and throughout the rest of the paper, a function $f:(a,b)\rightarrow\R$ is said to be absolutely continuous on $(a,b)$ iff there exists some $g\in L^1(a,b)$ such that $f(t)=f(b)-\int_t^bg(s)ds$ holds almost everywhere. In particular, $f^\prime=g$ almost everywhere and $f$ has a continuous representative that can be continuously extended to $[a,b]$.} on $[0,t_0]$ for any $x\in\M$ \textbf{and \boldmath{$\gamma<2$}}. 
\end{remark}

\begin{prop}\label{prop:energy-3}
Let $\psi$ be a smooth wave on a warped product spacetime $\left(\overline{M},\overline{g}\right)$ of type $-1$ with $\gamma\in(\nicefrac23,2)$. We define $\hat{\psi}(t,x):=\nicefrac{\psi(t,x)}{h(t)},$ $\ h(t)=\int_t^\infty a(s)^{-3}\,ds$. Then, for any $\epsilon>0$, there exists $t_0>0$ small enough such that, for
\[\beta_\epsilon=\max(6(\gamma-1)+\epsilon,2),\]
\[a(t)^{\beta_\epsilon}E\left(t,\hat{\psi}\right)\leq a(t_0)^{\beta_{\epsilon}}E\left(t_0,\hat{\psi}\right)\]
holds for any $0<t<t_0$. Additionally, for $\gamma=2$, the following estimate is satisfied for arbitrary $t_0>0$ and again any $0<t<t_0$:
\[a(t)^6E\left(t,\hat{\psi}\right)\leq a(t_0)^6E\left(t_0,\hat{\psi}\right)\]
\end{prop}

\begin{proof}
Once again, we straightforwardly calculate using $h=\int_t^\infty a(s)^{-3}\,ds,\,\dot{h}=-a^{-3}$:
\begin{align*}
\square_{\overline{g}}\hat{\psi}
=\,&2\frac{\dot{h}}{h}\del_t\hat{\psi}\numberthis\label{eq:rescaled-wave-op-hyp}
\end{align*}
In trying to analogize the proof of Proposition \ref{prop:energy-2} as much as possible, we will need to compare $\nicefrac{\dot{h}}{h}$ to $\nicefrac{\dot{a}}a$ for small times: We claim
\begin{equation}\label{eq:Hospital}
\lim_{t\to 0} \frac{\nicefrac{\dot{h}}h}{\nicefrac{\dot{a}}{a}}(t)=\frac{3\gamma}2-3
\end{equation}
for any $\gamma\in(\nicefrac23,2]$. First, we simplify the fraction:
\[\frac{\nicefrac{\dot{h}}h}{\nicefrac{\dot{a}}{a}}=\frac{-a^{-3}a}{\dot{a}h}=-\frac{\left(a^{2}\dot{a}\right)^{-1}}h\]
As $t\to 0$, the denominator diverges toward $\infty$ as shown in Lemma \ref{lem:scale-factor}. Regarding the numerator, the rephrased Friedman equation \eqref{eq:Friedmann-1a} with $\kappa=-1$ gives
\[a^{2}\dot{a}=a^2\sqrt{1+\frac{8\pi B}3a^{2-3\gamma}}=\sqrt{a^4+\frac{8\pi B}3a^{6-3\gamma}}\,.\]
With $a(0)=0$, this yields
\[\lim_{t\to 0}\left(a(t)^{2}\dot{a}(t)\right)^{-1}=\begin{cases}
\infty & \gamma<2\\
\sqrt{\frac3{8\pi B}} & \gamma=2
\end{cases}\,.\]
Thus, \eqref{eq:Hospital} already follows for $\gamma=2$. Else, we can apply the l'Hospital rule in step (A) to compute this limit (again recalling $\dot{h}=-a^{-3}$):
\begin{align*}
\lim_{t\to 0} \frac{\nicefrac{\dot{h}}h}{\nicefrac{\dot{a}}{a}}(t)=&-\lim_{t\to 0}\frac{a(t)^{-2}\dot{a}(t)^{-1}}{h(t)}\\
\equp{(A)}&-\lim_{t\to 0}\frac{-2a(t)^{-3}\dot{a}(t)\dot{a}(t)^{-1}-a(t)^{-2}\dot{a}(t)^{-2}\ddot{a}(t)}{-a(t)^{-3}}\\
=&-2-\lim_{t\to 0}\frac{a(t)\ddot{a}(t)}{\dot{a}(t)^2}\\
\equp{(B)}&-2-\lim_{t\to 0}\frac{a(t)\left(-\frac{4\pi}3(1+3(\gamma-1))\rho\right)}{1+\frac{8\pi B}3a(t)^{2-3\gamma}}\\
\equp{(C)}&-2-\lim_{t\to 0}\frac{a(t)^2\left(-\frac{4\pi}3(3\gamma-2)Ba(t)^{-3\gamma}\right)}{1+\frac{8\pi B}3a(t)^{2-3\gamma}}\\
=&-2+\lim_{t\to 0}\frac{\frac{4\pi B}{3}(3\gamma-2)}{a(t)^{3\gamma-2}+\frac{8\pi B}3}\\
=&-2+\frac12(3\gamma-2)=\frac{3\gamma}2-3
\end{align*}
We used the second Friedman equation \eqref{eq:Friedmann-2} with $p=(\gamma-1)\rho$ for (B) to substitute $\ddot{a}$ in the numerator, and \eqref{eq:Friedmann-1a} with $\kappa=-1$ to replace $\dot{a}$ in the denominator, as well as \eqref{eq:density-expansion-relation} to replace $\rho$ in (C). For the final limit, we recall that $3\gamma-2$ is positive for $\gamma>\nicefrac23$, that $a(0)=0$ holds and that $B$ is positive.\\

With this information in hand, we can now treat the energy as before: Using \eqref{eq:rescaled-wave-op-hyp} to replace $\del_t^2\hat{\psi}$, we calculate
\begin{align*}
\del_tE\left(t,\hat{\psi}\right)&=\int_{\M}\left(2\del_t^2\hat{\psi}\cdot\del_t\hat{\psi}-2\del_t\hat{\psi}\cdot a^{-2}\Lap\hat{\psi}-2\frac{\dot{a}}{a^3}\left\lvert\nabbar\hat{\psi}\right\rvert_{\g}^2\right)\,\vol{\M}\\
&=\int_{\M}\left(-\left(6\frac{\dot{a}}{a}+4\frac{\dot{h}}{h}\right)\left\lvert\del_t\hat{\psi}\right\rvert^2-2\frac{\dot{a}}{a}a^{-2}\left\lvert\nabbar\hat{\psi}\right\rvert_{\g}^2\right)\vol{\M}\\
&\geq-\max\left(6\frac{\dot{a}}{a}+4\frac{\dot{h}}{h},2\frac{\dot{a}}{a}\right)E(t,\hat{\psi})
\end{align*}
Now, it follows from \eqref{eq:Hospital} that, for any $\epsilon>0$, there exists some small enough $t_0>0$ such that, for all $0<t<t_0$,
$$\frac{\dot{h}(t)}{h(t)}\leq\left(\frac{3\gamma}2-3+\frac{\epsilon}4\right)\frac{\dot{a}(t)}{a(t)}$$
(since both $a$ and $\dot{a}$ are positive) and hence
\begin{align*}
\del_tE\left(t,\hat{\psi}\right)\geq&-\max\left(6+4\cdot\left(\frac{3\gamma}2-3+\frac{\epsilon}4\right),2\right)\frac{\dot{a}(t)}{a(t)}E\left(t,\hat{\psi}\right)\\
=&-\beta_\epsilon \frac{\dot{a}(t)}{a(t)}E\left(t,\hat{\psi}\right)\,.
\end{align*}
The stated energy estimate follows once again from a Gronwall argument. For $\gamma=2$, this works analogously, simply estimating
\[\del_tE(t,\psi)\geq -\max\left(6\frac{\dot{a}}{a}+4\frac{\dot{h}}{h},2\frac{\dot{a}}{a}\right)E(t,\hat{\psi})\geq -6\frac{\dot{a}}{a}E(t,\hat{\psi})\,,\]
since $\dot{h}=-a^{-3}<0$, $h>0$ and $\nicefrac{\dot{a}}{a}>0$, and then continuing as usual.
\end{proof}
In particular, we can derive the following pointwise estimate along the same lines as before:
\begin{corollary}\label{cor:rescaled-pointwise-est-hyp}
For $\left(\overline{M},\overline{g}\right)$, $\hat{\psi}$ and $\beta_\epsilon$ as in Proposition \ref{prop:energy-3} and $\nicefrac23<\gamma<2$, there exists $t_0>0$ small enough for any $\epsilon>0$ such that, for any $0<t<t_0$, the following pointwise estimate holds:
\begin{align*}
\left\lvert\Lap^N\hat{\psi}(t,\cdot)\right\rvert&\,\leq Ca(t_0)^{\frac{\beta_\epsilon}2}\left(\sqrt{E_N(t_0,\hat{\psi})}+\sqrt{E_{N+1}(t_0,\hat{\psi})}\right)\frac{t_0^{1-\nicefrac{\beta_\epsilon}{3\gamma}}-t^{1-\nicefrac{\beta_\epsilon}{3\gamma}}}{1-\nicefrac{\beta_\epsilon}{3\gamma}}\\
&\quad +\left\lvert\Lap^N\hat{\psi}(t_0,\cdot)\right\rvert
\end{align*}
For the stiff case ($\gamma=2$), one analogously obtains, again not requiring $t_0>0$ to be small here,
\begin{align*}
\left\lvert\Lap^N\hat{\psi}(t,\cdot)\right\rvert&\,\leq Ca(t_0)^3\left(\sqrt{E_N\left(t_0,\hat{\psi}\right)}+\sqrt{E_{N+1}\left(t_0,\hat{\psi}\right)}\right)\left(\log(t_0)-\log(t)\right)\\
&\quad + \left\lvert\Lap^N\hat{\psi}(t_0,\cdot)\right\rvert
\end{align*}
\end{corollary}

\begin{remark}\label{rem:energy-3-consequences}

Again, we turn to the question of whether the rescaled wave is absolutely continuous toward the Big Bang, which will help us answer whether we can extend it to the Big Bang hypersurface: If $\beta_\epsilon=2$, one has
\[1-\frac{\beta_\epsilon}{3\gamma}=1-\frac2{3\gamma}>0,\]
and else
\[1-\frac{\beta_\epsilon}{3\gamma}=1-\left(2\frac{\gamma-1}{\gamma}+\frac{\epsilon}{3\gamma}\right)=\frac2{\gamma}-1-\frac{\epsilon}{3\gamma}\]
is positive for small enough $\epsilon>0$ since $\frac2{\gamma}-1>0$ for $\gamma<2$. Hence, the proof once again even shows that $\hat{\psi}$ is absolutely continuous close to $t=0$, so $\displaystyle\lim_{t\to 0}\hat{\psi}(t,x)$ exists for any $x\in\M$,  \textbf{excluding the stiff case}.\\
Furthermore, it should be noted that this does not work for the stiff case since the upper estimate just obtained still diverges toward $\infty$ logarithmically when approaching $t=0$.
\end{remark}

\section{Global blow-up of waves}\label{sec:blow-up}

In this section, we will provide the proof of the main Theorem \ref{thm:main}, first treating the case $\gamma<2$ and proving the convergence in high regularity in type $0$, and then quickly arguing why type $-1$ follows completely analogously. Afterwards, we will turn to the stiff fluid case, going through both types there -- while the proof applied there is in principle also applicable to the previous setting, it only yields the asymptotic profile without additional strength of convergence and is thus treated separately.

\begin{proof}[Proof of Theorem \ref{thm:main} for $\gamma<2$]

First, let's turn to type $0$: Since, by Remark \ref{rem:energy-2-consequences},
\[t\in(0,t_0]\mapsto\frac{\Lap^N\psi(t,x)}{t^{1-\frac2{\gamma}}}\]
is absolutely continuous for any fixed $x\in\M$, with a time derivative that is integrable on $[0,t_0]$, $A_N(x):=\displaystyle\lim_{t\to 0}\frac{\Lap^N\psi(t,x)}{t^{1-\frac2{\gamma}}}$ exists for any $N\in\N$, $x\in\M$. \\

The argument for smoothness now works as follows: With the dominated convergence theorem, we can show that $A_N$ is in $L^2(\M)$ for any $N\in\N$. By choosing a sequence of smooth functions on $\M$ that approximates $A_N$ in $L^2$ and whose Laplacians approximate $A_{N+1}$ in $L^2$, it follows that $A_N$ is even in $H^2(\M)$, using ellipticity of $\Lap$, and thus continuous. Finally, we iterate this type of argument over $C^{2k}(\M)$ for $k\in\N$ to achieve arbitrarily high regularity, in particular for $A_0=A$.\\

To this end, we use the following notation: Choose an arbitrary decreasing sequence $\left(t_n\right)_{n\in\N}$ with $0<t_n\leq t_0$ for all $n\in\N$ and $t_n\rightarrow 0$ as $n\to\infty$. Further, define 
\[f_{N,n}(x):=\Lap^N\hat{\psi}(t_n,x)=\frac{\Lap^N\psi(t_n,x)}{t_n^{1-\frac2{\gamma}}}\,,\]
so $\left(f_{N,n}\right)_{n\in\N}$ converges to $A_N$ pointwise for any $N\in\N$. These sequences are \textit{consistent} in the sense that $\Lap f_{N,n}= f_{N+1,n}$ holds for all $n,N\in\N$.\\

\noindent By Lemma \ref{prop:energy-2}, $(t,x)\mapsto\left\lvert\frac{\Lap^N\psi(t,x)}{t^{1-\frac{2}{\gamma}}}\right\rvert^2$ is uniformly bounded on $[0,t_0]\times\M$ for any $N\in\N$. Since $\M$ is of finite volume, we can thus use the Dominated Convergence Theorem for $t$ approaching $0$ to deduce that $\left(f_{N,n}\right)_{n\in\N}$ converges to $A_N$ in $L^2\left(\M\right)$ for any $N\in\N$ as $n\to\infty$. By the consistency property $\Lap f_{N,n}=f_{N+1,n}$, it follows that this sequence must be a Cauchy sequence with regards to 
\[\|\Lap(\cdot)\|_{L^2\left(\M\right)}+\|\cdot\|_{L^2\left(\M\right)},\]
so also with regards to $\|\cdot\|_{H^2\left(\M\right)}$ by ellipticity. Thus, $(f_{N,n})_{n\in\N}$ converges in $H^2\left(\M\right)$, and this limit must obviously agree with $A_N$ almost everywhere, so $A_N\in H^2\left(\M\right)$. Furthermore, by the consistency property and uniqueness of weak derivatives, $\Lap A_N$ and $A_{N+1}$ must represent the same element of $L^2\left(\M\right)$.\\

Now, since we chose $(f_{N,n})_{n\in\N}$ to be consistent, it follows that $f_{N,n}$ and $\Lap f_{N,n}=f_{N,n+1}$ are Cauchy in $H^2(\M)$, so $\left(f_{N,n}\right)_{n\in\N}$ is a Cauchy sequence with regards to the norm 
\[\|\Lap(\cdot)\|_{H^2\left(\M\right)}+\|\cdot\|_{H^2\left(\M\right)}\]
for any $N\in\N$. By the standard Sobolev embedding $H^2(\M)\hookrightarrow C(\M)$, it is then also a Cauchy sequence with regards to
\[\|\Lap(\cdot)\|_{C\left(\M\right)}+\|\cdot\|_{C\left(\M\right)}\]
for any $N\in\N$, and thus a Cauchy sequence in $C^2\left(\M\right)$ by ellipticity of $\Lap$ (again see \cite[p.~463, Theorem 27]{Besse08}). Since the latter is a Banach space and any limit in $C^2\left(\M\right)$ must coincide with the pointwise limit, it follows that $A_N\in C^2\left(\M\right)$ must hold for any $N\in\N$. As $\Lap A_N=A_{N+1}$ holds in $L^2\left(\M\right)$, it must now also hold classically. Again using ellipticity, it now follows by the same approximation argument that $A_N\in C^4\left(\M\right)$ is satisfied for any $N$, and by iterating this argument that $A_N\in C^\infty\left(\M\right)$ must hold for any $N\in\N$. In particular, this shows that $A$ is smooth, $\Lap^NA_0=A_N$ and that $\Lap^N\hat{\psi}(t,\cdot)$ converges to $A_N$ in $C^{2k}(\M)$ for any $N,k\in\N$.\\

For type $-1$, Remark \ref{rem:energy-3-consequences} yields existence of $A_N$ along the same lines and fulfills the role of Remark \ref{rem:energy-2-consequences} in the rest of the proof as well. Besides replacing $t^{1-\nicefrac2{\gamma}}$ by $\int_t^\infty a(s)^{-3}\,ds$, everything else now follows identically since no (other) properties of the scale factor were used at any point.
\end{proof}

\begin{remark}
For $\gamma=2$, this argument fails in the first step since we do not have Remarks \ref{rem:energy-2-consequences} and \ref{rem:energy-3-consequences} at our disposal to even establish existence of $A$. Thus, we take a different route: By rearranging and integrating the wave equation along the same lines as for homogeneous waves in Remark \ref{rem:stationary-waves}, we see that $\psi$ takes precisely the desired form up to error terms that are either constant (and hence negligible compared to the divergent leading order) or an integral dependent on $a$ and $\Lap\psi$. Using the pointwise estimates from Section \ref{subsec:energy-estimates} to control $\Lap\psi$, we then show even this term to be bounded. In theory, we could have also used this strategy for $\gamma<2$, but we chose the strategy above since it essentially only relies on energy estimates and less on the structure of the wave equation itself that becomes more complicated in the full Einstein system.
\end{remark}
\begin{proof}[Proof of Theorem \ref{thm:main} for $\gamma=2$]
From the re-arranged wave equation in Corollary \autoref{cor:wave-formulas}, we have (since $a(t)>0$ is satisfied for all $t>0$)
\begin{equation}\label{eq:wave-rearrange}
\del_t\left(a^3\dot{\psi}\right)=a\Lap\psi\,.
\end{equation}
By integration, we obtain
\[\dot{\psi}(t,x)={a(t_0)^3}\dot{\psi}(t_0,x){a(t)^{-3}}-a(t)^{-3}\int_t^{t_0}a(s)\Lap\psi(s,x)\,ds\]
for some $t_0>0$.
Set $L=t_0$ for type $0$ and $L=\infty$ for type $-1$. Then, again by integration and re-arranging (at first only formally), one obtains
\begin{align*}
\psi(t,x)=\,&\psi(t_0,x)-a(t_0)^3\del_t\psi(t_0,x)\int_t^{t_0}a(s)^{-3}\,ds\\
&+\int_t^{t_0}a(s)^{-3}\left(\int_s^{t_0}a(r)\Lap\psi(r,x)\,dr\right)\,ds\\
=\,&\left(\int_t^La(s)^{-3}\,ds\right)\left(-a(t_0)^3\del_t\psi(t_0,x)+\int_0^{t_0}a(r)\Lap\psi(r,x)\,dr\right)\\
&-\left(\int_{t_0}^L a(s)^{-3}\,ds\right)\left(-a(t_0)^3\del_t\psi(t_0,x)+\int_0^{t_0}a(r)\Lap\psi(r,x)\,dr\right)+\psi(t_0,x)\\
& -\int_t^{t_0}\int_0^sa(s)^{-3}a(r)\Lap\psi(r,x)\,dr\,ds\,.\numberthis\label{eq:intermediate-stiff-asymp}
\end{align*}
Of course, this rearrangement is only allowed if $r\mapsto a(r)\Lap\psi(r,x)$ is integrable on $(0,t_0]$, which we will now verify: By Corollary \ref{cor:pointwise-estimate} for $N=1$, one knows that
\[\lvert\Lap\psi(r,x)\rvert\leq\lvert\Lap\psi(t_0,x)\rvert+Ca(t_0)^3\left(\int_r^{t_0}a(s)^{-3}\,ds\right)\left(\sqrt{E_1(t_0,\psi)}+\sqrt{E_2(t_0,\psi)}\right)\]
is satisfied for some $\g$-dependent constant $C$ (which may be suitably updated from line to line). For the sake of this argument, this information is simplified by working with the estimate
\[\lvert\Lap\psi(r,x)\rvert\leq C\left(1+\int_r^{t_0}a(s)^{-3}\,ds\right)\,.\]
By Lemma \ref{lem:scale-factor} with $\gamma=2$, one has $a(t)\simeq{t^\frac13}$ and hence $\int_t^{t_0}a(s)^{-3}ds=\O{\lvert\log(t)\rvert}$ for type $-1$ as $t\to 0$, and in type $0$ one even has $a(t)=t^{\frac13}$ for all $t>0$. Hence, one obtains the following (for w.l.o.g. small enough $t_0>0$ in type $-1$):
\begin{align*}
\int_s^{t_0}\left\lvert a(r)\Lap\psi(r,x)\right.\!&\left.\!\!\right\rvert dr\leq C \int_s^{t_0} r^{\frac13}(1+\lvert\log(r)\rvert)\,dr\\
&\leq C\left[\frac34\left(t_0^{\frac43}-s^{\frac43}+t_0^{\frac43}\lvert\log(t_0)\rvert+s^{\frac43}\lvert\log(s)\rvert\right)+\int_s^{t_0}\frac34r^{\frac13}\,dr\right]\\
&=C\left[\frac34\left(t_0^{\frac43}(1+\lvert\log(t_0)\rvert)+s^{\frac43}(-1+\lvert\log(s)\rvert)\right)+\frac9{16}\left(t_0^{\frac43}-s^{\frac43}\right)\right]
\end{align*}
As $s$ approaches 0, this remains bounded since $s^\alpha\lvert\log(s)\rvert\to 0$ as $s\to 0$ for any $\alpha>0$, so all our above calculations were justified. (Note that, for type $-1$, $L=\infty$ is allowed by Lemma \ref{lem:scale-factor}.)

First, we now finish type $-1$: As already implied by \eqref{eq:intermediate-stiff-asymp}, we set $A$ and $r$ as follows:
\begin{align*}
A(x)&:=-a(t_0)^3\del_t\psi(t_0,x)+\int_0^{t_0}a(r)\Lap\psi(r,x)\,dr\\
r(t,x)&:=\psi(t_0,x)-\left(\int_{t_0}^L a(s)^{-3}\,ds\right)\left(-a(t_0)^3\del_t\psi(t_0,x)+\int_0^{t_0}a(q)\Lap\psi(q,x)\,dq\right)\\
&\ -\int_t^{t_0}\int_0^sa(s)^{-3}a(q)\Lap\psi(q,x)\,dq\,ds
\end{align*}
Since $\psi$ and $a$ are smooth, so are $A$ and $r$. To prove the statement, it only needs to be shown that $r$ is bounded. Obviously, this only needs to be verified for the only non-constant term in the second line. We check, along similar lines to before, w.l.o.g. for $t_0>0$ small enough:
\begin{align*}
\left\lvert\int_t^{t_0}a(s)^{-3}\int_0^{s}a(r)\Lap\psi(q,x)\,dq\,ds\right\rvert
&\leq C\int_t^{t_0}\frac1s\int_0^{s}q^{\frac13}\left(1+\lvert\log(q)\rvert\right)\,dq\,ds\\
&= C\int_t^{t_0}\frac1s\left[\frac32s^{\frac43}+\frac9{16}s^{\frac43}+0\right]\,ds\\
&\leq C\left(t_0^{\frac43}-t^{\frac43}\right)
\end{align*}
Thus, $r$ remains bounded as $t\to 0$, in particular $r(t,x)=o(\lvert\log(t)\rvert)$ as $t\to 0$ and the asymptotic profile follows.\\
For type $0$, note that since $L=t_0$, the first summand in the second line of \eqref{eq:intermediate-stiff-asymp} vanishes, and one has
\[\int_t^L a(s)^{-3}\,ds=\int_t^{t_0}\frac1s\,ds=\log(t_0)-\log(t)\,.\]
Thus, \eqref{eq:intermediate-stiff-asymp} becomes
\begin{align*}
\psi(t,x)=\,&-\log(t)\left(-a(t_0)^3\del_t\psi(t_0,x)+\int_0^{t_0}a(r)\Lap\psi(r,x)\,dr\right)\\
&+\log(t_0)\left(-a(t_0)^3\del_t\psi(t_0,x)+\int_0^{t_0}a(r)\Lap\psi(r,x)\,dr\right)+\psi(t_0,x)\\
&\ -\int_t^{t_0}\int_0^sa(s)^{-3}a(r)\Lap\psi(r,x)\,dr\,ds
\end{align*}
and we analogously set
\begin{align*}
A(x)&:=a(t_0)^3\del_t\psi(t_0,x)-\int_0^{t_0}\int_s^{t_0}a(r)\Lap\psi(r,x)\,dr\,ds\,,\\
r(t,x)&:=\psi(t_0,x)+\log(t_0)\left(-a(t_0)^3\del_t\psi(t_0,x)+\int_0^{t_0}\int_s^{t_0}a(q)\Lap\psi(q,x)\,dq\,ds\right)\\
&\ -\int_t^{t_0}\int_0^sa(s)^{-3}a(q)\Lap\psi(q,x)\,dq\,ds\,.
\end{align*}
The argument now follows identically since the only term that is not obviously of order $o(\lvert\log(t)\rvert)$ approaching $0$ is the same one as in type $-1$, where all terms also have the same asymptotic behaviour.   
\end{proof}

\section{\todo{Sufficient conditions for highest order blow-up}}

\label{subsec:suff-cond}

In this final section, we will establish open conditions \todo{that ensure that  $A$ does not vanish pointwise, as well as some weaker conditions for global and pointwise blow-up that only work outside of the stiff case}. As indicated in the introduction, these essentially require the initial data to be velocity term dominated or close to that of a homogenous wave, respectively. For why \todo{the improved versions} fail in the stiff case, we point to Remark \ref{rem:failure}.\\

\subsection{\todo{A condition for pointwise blow-up}}

\begin{theorem}\label{thm:pointwise-blowup-stiff-incl}
\todo{Let $\psi_{\text{hom}}(x)=\mathcal{A}\cdot\int_t^{t_0}a(s)^{-3}\,ds$ be a homogeneous wave on a warped product spacetime of type $0$ or $-1$. Then, if $\psi$ satisfies the initial data assumption
\begin{align*}
\left\|a(t_0)^3\del_t\psi(t_0,\cdot)-\mathcal{A}\right\|_{C^0(M)}+\|\Lap\psi(t_0,\cdot)\|_{C^0(M)}\int_0^{t_0}a(s)ds&\\
+Ca(t_0)^3\left(\sqrt{E_1(t_0,\psi)}+\sqrt{E_2(t_0,\psi)}\right)\int_0^{t_0}\int_s^{t_0}a(r)^{-3}\,dr\,ds&<\mathcal{A}
\end{align*}
for $C>0$ as in Corollary \ref{cor:pointwise-estimate}, $A$ is nonvanishing.}
\end{theorem}
\begin{proof}
\todo{We again use \eqref{eq:wave-rearrange} as a starting point, but this time for the smooth wave $\psi-\psi_{\text{hom}}$. Since $a(t)^3\del_t(\psi-\psi_{\text{hom}})(t,x)\rightarrow A(x)-\mathcal{A}$ holds pointwise for any $x\in M$ as $t\to 0$, we have
\begin{equation*}
\left\lvert a(t_0)^3\del_t(\psi-\psi_{\text{hom}})(t_0,x)-(A(x)-\mathcal{A})\right\rvert\leq \int_0^{t_0}a(s)\lvert\Lap(\psi-\psi_{\text{hom}})(s,x)\rvert\,ds\,.
\end{equation*}
Using Corollary \ref{cor:pointwise-estimate} to estimate integrand uniformly in $x$ and rearranging the inequality, we get
\begin{align*}
\lvert A(x)\rvert\geq&\,\lvert\mathcal{A}\rvert-\left\lvert a(t_0)^3\del_t(\psi-\psi_{\text{hom}})(t_0,x)\right\rvert\\
&\,-\left\lvert a(t_0)^3\del_t(\psi-\psi_{\text{hom}})(t_0,x)-(A(x)-\mathcal{A})\right\rvert\\
&\,-\|\Lap(\psi-\psi_{\text{hom}})(t_0,\cdot)\|_{C^0(M)}\int_0^{t_0}a(s)\,ds\\
&\,-Ca(t_0)^3\left(\sqrt{E_1(t_0,\psi-\psi_{\text{hom}})}+\sqrt{E_2(t_0,\psi-\psi_{\text{hom}})}\right)\cdot\\
&\quad\quad\quad\cdot\int_0^{t_0}a(s)\int_s^{t_0}a(r)^{-3}\,dr\,ds
\end{align*}
and the statement now follows upon realising that $a(t_0)^3\del_t\psi_{\text{hom}}(t_0,\cdot)=\mathcal{A}$ holds by definition and that $\psi_{\text{hom}}$ drops out in all subsequent terms precisely because it is spatially homogenous.}
\end{proof}

\begin{corollary}\label{cor:pointwise-blowup-stiff-incl}
\todo{If, for $t_0>0$ small enough in type $-1$, the initial data is controlled in the sense that, for some $\epsilon>0$ such that, 
\begin{equation*}
\|a(t_0)^3\del_t\psi(t_0,\cdot)-\mathcal{A}\|_{C^0(M)}\lesssim\epsilon
\end{equation*}
and
\begin{equation}\label{eq:pw-inprovement-cond}
g(t_0)\left(\|\del_t\Lap(\psi)(t_0,\cdot)\|_{H^2(M)}+\|\Lap\psi\|_{H^3(M)}\right)+t_0^{1+\frac2{3\gamma}}\|\Lap\psi\|_{C^0(M)}\lesssim\epsilon
\end{equation}
hold, where
\begin{equation*}
g(t_0)=\begin{cases}
\left(\left(2-\frac4{3\gamma}\right)^{-1}+\left(1+\frac2{3\gamma}\right)^{-1}\right)t_0^{2-\frac4{3\gamma}} & \gamma<2\\[1em]
t_0^\frac73\left(1+\lvert\log(t_0)\rvert\right) & \gamma=2\,,
\end{cases}
\end{equation*}
then $\lvert A\rvert>0$ holds if $\nicefrac{\epsilon}{\lvert\mathcal{A}\rvert}>0$ is small enough.}
\end{corollary}
\begin{proof}
\todo{To illustrate why this translates to our simplified statement, we look at type $-1$ first and take $t_0$ to be small enough so that we can estimate $a(t)\leq C_a\cdot t^{\frac2{3\gamma}}$. We then see
\begin{equation*}
0\leq\int_0^{t_0}a(s)\,ds\leq C_a \int_0^{t_0}s^{\frac2{3\gamma}}\,ds= C_a\left(1+\frac2{3\gamma}\right)^{-1}t_0^{1+\frac2{3\gamma}}
\end{equation*}
and, for $\gamma<2$
\begin{align*}
0\leq a(t_0)^3\int_0^{t_0}a(s)\int_s^{t_0}a(r)^{-3}\,dr\,ds\leq&\,C_at_0^{\frac2{3\gamma}}\int_0^{t_0}s^{\frac2{3\gamma}}\left(s^{1-\frac2{\gamma}}-t_0^{1-\frac2{\gamma}}\right)\,ds\\
&\,=C_a\left(\left(2-\frac4{3\gamma}\right)^{-1}+\left(1+\frac2{3\gamma}\right)^{-1}\right)t_0^{2-\frac4{3\gamma}},
\end{align*}
while we analogously get for $\gamma=2$
\begin{equation*}
0\leq a(t_0)^3\int_0^{t_0}a(s)\int_s^{t_0}a(r)^{-3}\,dr\,ds\leq C_at_0^\frac73\left(\frac9{16}+\frac14\lvert\log(t_0)\rvert\right)
\end{equation*}
For type $0$, all of these calculations hold with $C_a$ equal to 1. Further, we observe that the energies occuring in Theorem \ref{thm:pointwise-blowup-stiff-incl} can be controlled by the Sobolev norms scaled by $g$ in \eqref{eq:pw-inprovement-cond}. Altogether, in light of our initial data assumption, it is now sufficient to satisfy
\[\epsilon<K\lvert\mathcal{A}\rvert\]
for some $K>0$ that is independent of $\epsilon$ and $\mathcal{A}$, and thus choosing $\nicefrac{\epsilon}{\lvert\mathcal{A}\rvert}$ to be small enough is sufficient.}
\end{proof}

\subsection{Improved conditions outside of the stiff case}
\todo{The fundamental difference between warped products associated to $\gamma=2$ and to $\gamma<2$ is that we can improve the convergence result in Theorem \ref{thm:main} for $\gamma<2$ to convergence within our energies:}

\begin{prop}\label{thm:energy-convergence}
In type $-1$ warped products with $\gamma<2$, the following holds denoting $h(t)=\int_t^{\infty}a(s)^{-3}\,ds$:
\[a(t)^6E_N(t,{\psi}-Ah)\rightarrow 0\ \text{as}\ t\to 0\]
For type $0$ warped products with $\gamma<2$, one analogously has
\[\lim_{t\to 0}t^{\frac4{\gamma}}E_N\left(t,\psi(t,\cdot)-A\cdot t^{1-\frac2{\gamma}}\right)=0.\]
\end{prop}
\begin{proof}
We only prove the former estimate, since the proof of the latter is analogous and simpler. One calculates:
\begin{align*}
&a(t)^6E\left(t,\psi-Ah\right)\\
=\,&a(t)^6\int_{\M}\left[\left\lvert h(t)\del_t\hat{\psi}(t,\cdot)-a(t)^{-3}\hat{\psi}(t,\cdot)+a(t)^{-3}A\right\rvert^2\right.+\\
&\left.\phantom{\left\lvert h(t)\del_t\hat{\psi}(t,\cdot)-\frac{a(t)^{-3}}{h(t)}\psi(t,\cdot)+a(t)^{-3}A\right\rvert^2}+a(t)^{-2}h(t)^2\left\lvert\nabbar\hat{\psi}(t,\cdot)-\nabbar A\right\rvert_{\g}^2\right]\vol{\M}\\
\leq\,&2a(t)^6\int_{\M}\left[h(t)^2\left\lvert\del_t\hat{\psi}(t,\cdot)\right\rvert^2+a(t)^{-6}\left\lvert\hat{\psi}(t,\cdot)-A\right\rvert^2+\right.\\
&\left.\phantom{h(t)^2\left\lvert\del_t\hat{\psi}(t,\cdot)\right\rvert^2+a(t)^{-6}\left\lvert\hat{\psi}(t,\cdot)-A\right\rvert^2}+a(t)^{-2}h(t)^2\left(\left\lvert\nabbar\hat{\psi}(t,\cdot)\right\rvert_{\g}^2+\left\lvert\nabbar A\right\rvert_{\g}^2\right)\right]\vol{\M}\\
\leq\,&2a(t)^6h(t)^2E\left(t,\hat{\psi}\right)+2\int_{\M}\left\lvert\hat{\psi}(t,\cdot)-A\right\rvert^2\vol{\M}+2{h(t)^2}{a(t)^4}\int_{\M}\left\lvert\nabbar A\right\rvert_{\g}^2\vol{\M} \numberthis\label{eq:energy-convergence-all-terms}
\end{align*}
Now, we analyse all three terms as $t\to 0$:\\
Regarding the first term, we have shown in Lemma \ref{lem:scale-factor} that $a(t)=\O{t^{\nicefrac2{3\gamma}}}$ and $h(t)=\O{t^{1-\nicefrac2{\gamma}}}$. Thus, $a(t)^6h(t)^2=\O{t^2}$. On the other hand, combining Proposition \ref{prop:energy-3} and again Lemma \ref{lem:scale-factor} yields for arbitrarily small $\epsilon>0$ as long $t_0>t>0$ small enough:
\[E\left(t,\hat{\psi}\right)\leq E\left(t_0,\hat{\psi}\right)\left(\frac{a(t_0)}{a(t)}\right)^{\beta_{\epsilon}}\leq E\left(t_0,\hat{\psi}\right)a(t_0)^{\beta_{\epsilon}}\cdot Ct^{-\nicefrac{2\beta_\epsilon}{3\gamma}}\]
If $\beta_{\epsilon}=2$, one has $-\nicefrac{2\beta_{\epsilon}}{3\gamma}=-\nicefrac{4}{3\gamma}>-2$. Else, one has
\[-\frac{2\beta_{\epsilon}}{3\gamma}=-\frac2{3\gamma}\left(6(\gamma-1)+\epsilon\right)=\frac4{\gamma}-2-\frac{\epsilon}{3\gamma}\]
For $0<\epsilon<3\gamma\left(\nicefrac4\gamma-2\right)=12-6\gamma$, one can ensure that this is positive (recalling $\gamma<2$). Hence, one deduces that $E(t,\psi)=\O{t^{-2+\delta}}$ holds for some $\delta>0$ in any case and thus the first summand vanishes.\\
The second term simply vanishes by the Dominated Convergence Theorem.\\
Regarding the final term, one has by Lemma \ref{lem:scale-factor} that
\[{h(t)^2}{a(t)^4}=\O{t^{2-\frac4{\gamma}+\frac8{3\gamma}}}=\O{t^{2-\frac{4}{3\gamma}}},\]
so this factor converges to $0$ as $t\to 0$ since $\gamma>\nicefrac23$. Since $A$ is smooth, the integral is finite and this term as a whole converges to $0$.\\
Altogether, the entire right hand side of \eqref{eq:energy-convergence-all-terms} now vanishes in the limit, proving the statement.
\end{proof}

Before establishing our blow-up conditions, we quickly collect the following lemma:

\begin{lemma}\label{lem:H1-est}
For any smooth wave $\psi$ on a warped product spacetime as in Proposition \ref{prop:energy-1} and any $0<t<t_0$, the following holds:
\begin{align*}
\sqrt{\int_{\M}\left\lvert\nabbar\psi(t,\cdot)\right\rvert^2_{\g}\vol{\M}}&\,\leq\sqrt{2}\sqrt{E(t_0,\psi)+E_1(t_0,\psi)}\int_t^{t_0}\frac{a(t_0)^3}{a(s)^{3}}\,ds\\
&\quad + \sqrt{\int_{\M}\left\lvert\nabbar\psi(t_0,\cdot)\right\rvert^2_{\g}\vol{\M}}
\end{align*}
\end{lemma}
\begin{proof}
For the sake of convenience, we denote $F(t,\psi):=\sqrt{\displaystyle\int_{\M}\left\lvert\nabbar\psi(t,\cdot)\right\rvert^2_{\g}\vol{\M}}$. One calculates for $0<s<t_0$:
\begin{align*}
-\frac12\left(\del_t\left(F(\cdot,\psi)^2\right)\right)(s)=\,&\int_{\M}-\g\left(\nabbar\psi(s,\cdot),\del_t\nabbar\psi(s,\cdot)\right)\,\vol{\M}\\
\leq\,&\sqrt{\int_{\M}\left\lvert\nabbar\psi(s,\cdot)\right\rvert_{\g}^2\,\vol{\M}}\sqrt{\int_{\M}\left\lvert\nabbar\del_t\psi(s,\cdot)\right\rvert_{\g}^2\,\vol{\M}}\\
\leq\,& F(s,\psi)\sqrt{\frac12\int_{\M}\left\lvert\del_t\psi(s,\cdot)\right\rvert^2+\left\lvert\del_t\Lap\psi(s,\cdot)\right\rvert^2\,\vol{\M}}\\
\leq\,& F(s,\psi)\sqrt{\frac{E(s,\psi)+E_1(s,\psi)}2}\\
\leq\,&F(s,\psi)\sqrt{\frac{E(t_0,\psi)+E_1(t_0,\psi)}2}\frac{a(t_0)^3}{a(s)^3}
\end{align*}
On the other hand, one has $\frac12\left(\del_t\left(F(\cdot,\psi)^2\right)\right)(s)=F(s,\psi)\cdot\del_tF(s,\psi)$. Hence,
\[-\del_t F(s,\psi)\leq \sqrt{2}\sqrt{E(t_0,\psi)+E_1(t_0,\psi)}\frac{a(t_0)^3}{a(s)^3}\]
and thus the statement follows from integration on $s\in[t,t_0]$.
\end{proof}

\begin{lemma}\label{lem:limit-equality} For type $0$ warped product spacetimes with $\gamma<2$, one has
\[\lim_{t\downarrow 0}a(t)^6E(t,\psi)=\lim_{t\downarrow 0}t^{\frac4{\gamma}}E(t,\psi)=\left(1-\frac2{\gamma}\right)^2\int_{\M}\lvert A\rvert^2\,\vol{\M}\,.\]
For type $-1$ with $\gamma<2$, the following holds:
\[\lim_{t\to 0}a(t)^6E(t,\psi)=\int_{\M}\lvert A\rvert^2\vol{\M}\]
\end{lemma}
\begin{proof}
As earlier, we only prove the type $-1$ case, denote $h(t)=\int_t^{\infty}a(s)^{-3}\,ds$ and then calculate:
\begin{align*}
a(t)^6E(t,\psi)=\,&a(t)^6\int_{\M}\left[\left\lvert\del_t\left(\psi-Ah\right)-a(t)^{-3}\cdot A\right\rvert^2+a(t)^{-2}\left\lvert\nabbar\left(\psi-Ah\right)\right\rvert^2_{\g}\right.\\
&\quad\quad\quad\quad\quad\quad\left. +2a(t)^{-2}h(t)\g\left(\nabbar\psi,\nabbar A\right)-a(t)^{-2}h(t)^2\left\lvert\nabbar A\right\rvert^2_{\g}\right]\vol{\M}\\
=\,&a(t)^6E(t,\psi-Ah)-2a(t)^3\int_{\M}A\cdot\del_t\left(\psi-Ah\right)\,\vol{\M}+\int_{\M}\lvert A\rvert^2\,\vol{\M}\\
&\quad -a(t)^4h(t)^2\int_{\M} \left[\frac{\psi}{h}\cdot\Lap A+\left\lvert\nabbar A\right\rvert^2_{\g}\right]\vol{\M} \numberthis \label{eq:limit-equality-terms-listed}
\end{align*}
The first term vanishes by Proposition \ref{thm:energy-convergence}, and so does the second one since
\[a(t)^3\left\lvert\int_{\M}A\cdot\del_t(\psi-Ah)\vol{\M}\right\rvert\leq\|A\|_{L^2\left(\M\right)}\sqrt{a(t)^6E(t,\psi-Ah)}\longrightarrow 0.\]
Regarding the final term, $\nicefrac\psi{h}=\hat{\psi}$ converges to $A$ pointwise by definition and is uniformly bounded by Remark \ref{rem:energy-3-consequences}, so the integral remains finite in the limit by the Dominated Convergence Theorem (it even vanishes after integration by parts). Furthermore, by Lemma \ref{lem:scale-factor}, as $t$ approaches $0$, $a(t)^4=\O{t^{\frac8{3\gamma}}}$ and $h(t)^2=\O{t^{2-\frac4{\gamma}}}$. Hence, the prefactor asymptotically behaves like $t^{2-\frac4{3\gamma}}$ and in particular converges to zero, so the entire summand does as well. Since all terms beside $\|A\|_{L^2\left(\M\right)}^2$ now vanish in the limit, the statement follows.
\end{proof}

With these lemmata now in hand, we can use the previous energy estimates to construct sufficient conditions that $A$ does not vanish entirely. For the sake of simplicity, we first start out with type $0$ and adjust the statement and proof for type $-1$ afterward.

\begin{theorem}\label{thm:Global-Blowup}
Suppose that, over a type $0$ warped product spacetime $\left(\overline{M},\overline{g}\right)$ associated with $\gamma<2$, for sufficiently small $t_0>0$, $\del_t\psi(t_0,\cdot)$ is not identically zero and there exists some $\epsilon\in(0,1)$  such that
\begin{equation}\label{eq:global-blowup-req1}
\epsilon\left[1-Gt_0^{2-\frac{4}{3\gamma}}\right]\int_{\M}\left\lvert\del_t\psi(t_0,\cdot)\right\rvert^2\,\vol{\M}>Gt_0^{2-\frac4{3\gamma}}\int_{\M}\left\lvert\del_t\Lap\psi(t_0,\cdot)\right\rvert^2\,\vol{\M}
\end{equation}
and
\begin{align*}
(1-\epsilon)&\left[1-Gt_0^{2-\frac{4}{3\gamma}}\right]a(t_0)^2\int_{\M}\left\lvert\del_t\psi(t_0,\cdot)\right\rvert^{2}\,\vol{\M}> \\
&>\left(1+Gt_0^{2-\frac4{3\gamma}}\right)\int_{\M}\left\lvert\nabbar\psi(t_0,\cdot)\right\rvert_{\g}^2\,\vol{\M}+Gt_0^{2-\frac4{3\gamma}}\int_{\M}\left\lvert\nabbar\Lap\psi(t_0,\cdot)\right\rvert^2_{\g}\,\vol{\M} \numberthis \label{eq:global-blowup-req2}
\end{align*}
hold, where $G:=\frac{32}{3\gamma\left(1-\frac2{\gamma}\right)^2}\left(\frac{3\gamma}8-\frac2{1+\frac2{3\gamma}}+\frac1{2-\frac4{3\gamma}}\right)=\frac{4}{1-\left(\frac2{3\gamma}\right)^2}>0$. Then $\|A\|_{L^2\left(\M\right)}>0$.
\end{theorem}
\begin{proof}
Applying the results \eqref{eq:divergence-energy-flux} and \eqref{eq:key-divergence-theorem} from the energy-flux approach to the original energy estimates, it follows that
\[a(t)^6E(t,\psi)=a(t_0)^{6}E(t_0,\psi)-4\int_{t}^{t_0}\int_{\M_s}\dot{a}(s)\left\lvert\nabbar\psi(s,\cdot)\right\rvert_{\g}^2\vol{\M_s}ds\]
Thus, recalling $\vol{\M_s}=a(s)^3\vol{\M}$ and using Lemma \ref{lem:H1-est} for a lower bound, these estimates follow:
\begin{align*}
a(t)^6E(t,\psi)
\geq\,&a(t_0)^6 E(t_0,\psi)-\int_t^{t_0}4\dot{a}(s)a(s)^3\left(\sqrt{\int_{\M}\left\lvert\nabbar\psi(t_0,\cdot)\right\rvert_{\g}^2\vol{\M}}\,+\right.\\
&\left.+\sqrt{2}\sqrt{E(t_0,\psi)+E_1(t_0,\psi)}\cdot \int_s^{t_0}\frac{a(t_0)^3}{a(r)^3}\,dr\right)^2\,ds\\
\geq\,& a(t_0)^{6}E(t_0,\psi)-8\left(\int_{\M}\left\lvert\nabbar\psi(t_0,\cdot)\right\rvert_{\g}^2\,\vol{\M}\right)\int_t^{t_0}\dot{a}(s)a(s)^3\,ds\\
&-16a(t_0)^6[E(t_0,\psi)+E_1(t_0,\psi)]\int_t^{t_0}\dot{a}(s)a(s)^3\left(\int_s^{t_0}a(r)^{-3}\,dr\right)^2\,ds\\
=\,&a(t_0)^{6}E(t_0,\psi)-2\left(\int_{\M}\left\lvert\nabbar\psi(t_0,\cdot)\right\rvert_{\g}^2\,\vol{\M}\right)\left(a(t_0)^4-a(t)^4\right)\\
&-16a(t_0)^6[E(t_0,\psi)+E_1(t_0,\psi)]\int_t^{t_0}\dot{a}(s)a(s)^3\left(\int_s^{t_0}a(r)^{-3}\,dr\right)^2\,ds \numberthis \label{eq:intermediate-blowup-est}
\end{align*}
By Lemma \ref{lem:limit-equality}, the left hand side converges to $\left(1-\nicefrac2{\gamma}\right)^2\|A\|_{L^2\left(\M\right)}^2$, so it only needs to be shown that the right hand side is strictly greater than zero as $t\to 0$. One quickly collects
\[\left(\int_s^{t_0}a(r)^{-3}\,dr\right)^2=\left(\int_s^{t_0}r^{\frac2\gamma}\,dr\right)^2=\left(\frac{t_0^{1-\frac2\gamma}-s^{1-\frac2\gamma}}{1-\frac2\gamma}\right)^2\,\]
and
\begin{align*}
\dot{a}(s)a(s)^3\left(\int_s^{t_0}a(r)^{-3}\,dr\right)^2&=\frac2{3\gamma}s^{\left(\frac2{3\gamma}-1\right)+\frac2{\gamma}}\left(\frac{t_0^{1-\frac2\gamma}-s^{1-\frac2\gamma}}{1-\frac2\gamma}\right)^2\\
&=\frac2{3\gamma\left(1-\frac2{\gamma}\right)^2}\left(s^{\frac8{3\gamma}-1}t_0^{2-\frac4{\gamma}}-2s^{\frac2{3\gamma}}t_0^{1-\frac2\gamma}+s^{1-\frac4{3\gamma}}\right)\,.
\end{align*}
After taking the limit $t\to 0$, the right hand side of \eqref{eq:intermediate-blowup-est} now, using the above formula to simplify the final term, becomes
\begin{align*}
&\,t_0^{\frac4{\gamma}}E(t_0,\psi)-2t_0^{\frac8{3\gamma}}\int_{\M}\left\lvert\nabbar\psi(t_0,\cdot)\right\rvert_{\g}^2\,\vol{\M}\\
&-\frac{32}{3\gamma}\frac{t_0^{\frac4{\gamma}}(E(t_0,\psi)+E_1(t_0,\psi))}{\left(1-\frac2{\gamma}\right)^2}\left(\frac{3\gamma}8-\frac2{1+\frac2{3\gamma}}+\frac1{2-\frac4{3\gamma}}\right)t_0^{2-\frac4{3\gamma}}\\
=\,&t_0^{\frac4{\gamma}}E(t_0,\psi)-2t_0^{\frac8{3\gamma}}\int_{\M}\left\lvert\nabbar\psi(t_0,\cdot)\right\rvert_{\g}^2\,\vol{\M}-G(E(t_0,\psi)+E_1(t_0,\psi))t_0^{2+\frac8{3\gamma}}\\
=\,&t_0^{\frac4{\gamma}}\left(1-Gt_0^{2-\frac4{3\gamma}}\right)\int_{\M}\lvert\del_t\psi(t_0,\cdot)\rvert^2\vol{\M}\\
&-Gt_0^{\frac4{\gamma}}t_0^{2-\frac4{3\gamma}}\int_{\M}\lvert\del_t\Lap\psi(t_0,\cdot)\rvert^2\vol{\M}\\
&-t_0^{\frac8{3\gamma}}\left(1+Gt_0^{2-\frac4{3\gamma}}\right)\int_{\M}\left\lvert\nabbar\psi(t_0,\cdot)\right\rvert^2_{\g}\vol{\M}-Gt_0^{\frac8{3\gamma}}t_0^{2-\frac4{3\gamma}}\int_{\M}\left\lvert\nabbar\Lap\psi(t_0,\cdot)\right\rvert_{\g}^2\,\vol{\M}\,.
\end{align*}
One now easily checks that if the conditions \eqref{eq:global-blowup-req1} and \eqref{eq:global-blowup-req2} are satisfied, this is positive, i.e.~$\|A\|_{L^2\left(\M\right)}>0$.\\
Finally, the simplification of $G$ is just a straightforward calculation.
\end{proof}

\begin{theorem}\label{thm:Global-Blowup-hyp}
Over type $-1$ warped product spacetimes $\left(\overline{M},\overline{g}\right)$ with $\gamma<2$, for sufficiently small $t_0>0$ and assuming $\del_t\psi(t_0,\cdot)$ is not identically zero, the equivalent statement to Theorem \ref{thm:Global-Blowup} holds when replacing $G$ with a suitably large constant $\tilde{G}$.
\end{theorem}
\begin{proof}[Proof of Theorem \ref{thm:Global-Blowup-hyp}]
Up to \eqref{eq:intermediate-blowup-est}, the proof is identical to the type $0$ setting, where the limit of the left hand side even converges precisely to $\|A\|_{L^2\left(\M\right)}^2$ by Lemma \ref{lem:limit-equality}. The only thing that needs to be done is to track how $a$ only behaving like $t^{\frac2{3\gamma}}$ \textit{asymptotically} influences the terms on the right hand side of \eqref{eq:intermediate-blowup-est}. Note that, by Lemma \ref{lem:scale-factor}, suitable $k_1<1<k_2$ and exist for $t_0>0$ small enough such that
\begin{equation*}
k_1t^{\frac2{3\gamma}}\leq a(t)\leq k_2t^{\frac2{3\gamma}}
\end{equation*}
is satisfied for all $0<t<t_0$. Furthermore, one obtains for some $k_3>0$:
\[0\leq\dot{a}(t)=\sqrt{\frac{8\pi B}3 a(t)^{2-3\gamma}+1}\leq\sqrt{k_1^{2-3\gamma}}\sqrt{\frac{8\pi B}3 t^{\frac4{3\gamma}-2}+1}\leq k_3 t^{\frac2{3\gamma}-1}\]

Hence, one checks:
\[\int_t^{t_0}\dot{a}(s)a(s)^3\left(\int_s^{t_0}a(r)^{-3}\,dr\right)^2ds\leq \frac{k_3k_2^3}{k_1^6}\int_t^{t_0}\frac2{3\gamma}s^{\frac2{3\gamma}-1}s^{\frac2{\gamma}}\left(\int_s^{t_0}r^{-\frac2\gamma}\,dr\right)^2\,ds\]
Setting
\[\tilde{G}=G\cdot\frac{k_3k_2^3}{k_1^6},\]
one now performs precisely the same calculations as in type $0$ on these terms and the right hand side of \eqref{eq:intermediate-blowup-est} becomes
\begin{align*}
&\,a(t_0)^6E(t_0,\psi)-2a(t_0)^4\int_{\M}\left\lvert\nabbar\psi(t_0,\cdot)\right\rvert_{\g}^2\,\vol{\M}\\
&-\frac{32}{3\gamma}\frac{a(t_0)^6(E(t_0,\psi)+E_1(t_0,\psi))}{\left(1-\frac2{\gamma}\right)^2}\left(\frac{3\gamma}8-\frac2{1+\frac2{3\gamma}}+\frac1{2-\frac4{3\gamma}}\right)\frac{k_3k_2^3}{k_1^6}t_0^{2-\frac4{3\gamma}}\\
=&\,a(t_0)^6E(t_0,\psi)-2a(t_0)^4\int_{\M}\left\lvert\nabbar\psi(t_0,\cdot)\right\rvert_{\g}^2\,\vol{\M}-\tilde{G}a(t_0)^6t_0^{2-\frac4{3\gamma}}\left(E(t_0,\psi)+E_1(t_0,\psi\right))\\
=&\,a(t_0)^6\left(1-\tilde{G}t_0^{2-\frac4{3\gamma}}\right)\int_{\M}\lvert\del_t\psi(t_0,\cdot)\rvert^2\vol{\M}\\
&-\tilde{G}t_0^{2-\frac4{3\gamma}}a(t_0)^6\int_{\M}\lvert\del_t\Lap\psi(t_0,\cdot)\rvert^2\vol{\M}\\
&-a(t_0)^4\left[\left(1+\tilde{G}t_0^{2-\frac4{3\gamma}}\right)\int_{\M}\left\lvert\nabbar\psi(t_0,\cdot)\right\rvert^2_{\g}\vol{\M}-\tilde{G}t_0^{2-\frac4{3\gamma}}\int_{\M}\left\lvert\nabbar\Lap\psi(t_0,\cdot)\right\rvert_{\g}^2\,\vol{\M}\right]
\end{align*}
Again, one now just checks \eqref{eq:global-blowup-req1} and \eqref{eq:global-blowup-req2}, with $G$ replaced by $\tilde{G}$, to ensure that this is strictly larger than zero, proving the statement.
\end{proof}

Finally, we can also formulate \todo{an improved} ($\left(\M,\g\right)$-dependent) criterion on whether $A$ is pointwise non-vanishing:

\begin{theorem}\label{thm:pointwise-blowup}
Consider a warped product spacetime $\left(\overline{M},\overline{g}\right)$ of type $0$ or $-1$ with $\gamma<2$. Let $K>0$ be such that \[\|\phi\|_{C\left(\M\right)}^2\leq K^2\left(\|\phi\|^2_{L^2\left(\M\right)}+\|\Lap\phi\|^2_{L^2\left(\M\right)}\right)\] for all $\phi\in C^\infty\left(\M\right)$.  Further, let $\epsilon>0$, $\lvert\mathcal{A}\rvert>\frac{K}{1-\nicefrac2{\gamma}}\epsilon$  (resp. $\lvert\mathcal{A}\rvert>K\epsilon$) for type $0$ (resp. type $-1$) and ${\psi}_{\text{hom}}(t,x):=\mathcal{A}\cdot t^{1-\nicefrac2{\gamma}}$ (resp. ${\psi}_{\text{hom}}(t,x)=C\cdot h(t)=\mathcal{A}\cdot\int_t^\infty a(s)^{-3}\,ds$) be homogeneous waves. Then, if
\[a(t_0)^6\left[E\left(t_0,\psi-\psi_{\text{hom}}\right)+E\left(t_0,\Lap\left(\psi-{\psi}_{\text{hom}}\right)\right)\right]\leq\epsilon^2\]
holds for some $t_0>0$, $A$ is non-vanishing.
\end{theorem}
\begin{proof}
Only type $-1$ will be proven since type $0$ follows identically, exchanging $h(t)$ with $t^{1-\frac2{\gamma}}$ and adapting for the differences in scaling that causes.\\
First, note that a suitable $K>0$ exists since $\Lap$ is an elliptic operator of second order. Further, ${\psi}_{\text{hom}}$ and hence also $\psi-{\psi}_{\text{hom}}$ are smooth waves. In particular, we obtain
\[\|A-\mathcal{A}\|_{L^2\left(\M\right)}^2=\lim_{t\to0}a(t)^6E(t,\psi-{{\psi}_{\text{hom}}})\leq a(t_0)^6E(t_0,\psi-{\psi}_{\text{hom}}).\]
By Theorem \ref{thm:main} for $\gamma<2$, $\Lap\left(\frac{\psi-\overline{\psi}}{t^{1-\frac2{\gamma}}}\right)\rightarrow \Lap(A-\mathcal{A})=\Lap A$
holds as $t\to 0$ since $\frac{(\psi-\psi_{\text{hom}})(t,\cdot)}{h(t)}$ converges to $A-\mathcal{A}$ in $C^2\left(\M\right)$, and we obtain with Proposition \ref{prop:energy-1}:
\begin{align*}
&\,\left(\|A-\mathcal{A}\|_{L^2\left(\M\right)}^2+\|\Lap(A-\mathcal{A})\|_{L^2\left(\M\right)}^2\right)\\
\leq&\,a(t_0)^{6}\left[E\left(t_0,\psi-{\psi}_{\text{hom}}\right)+E_1\left(t_0,\psi-{\psi}_{\text{hom}}\right)\right]\leq\epsilon^2
\end{align*}
By definition of $K$, it now follows that, for any $x\in\M$,
\[\lvert A(x)-\mathcal{A}\rvert^2\leq K^2\left(\|A-\mathcal{A}\|_{L^2\left(\M\right)}^2+\|\Lap\left(A-\mathcal{A}\right)\|_{L^2(\M)}^2\right)\leq K^2\epsilon^2\]
and thus by assumption
\[\lvert A(x)\rvert\geq \lvert \mathcal{A}\rvert-K\epsilon>0\,.\]
\end{proof}

\begin{remark}\label{rem:failure}
One cannot \todo{extend the weaker blow-up criteria within this final subsection to the stiff case, or at least not with a similar approach}, which we will now quickly illustrate for the framework associated with $\kappa=-1$ (similar issues occur in the setting associated with flat space): Referring to \eqref{eq:energy-convergence-all-terms} and looking at the first term on the right hand side, one sees that the energy convergence now no longer holds since, by Proposition \ref{prop:energy-3}, one can only obtain
\[a(t)^6h(t)^2E\left(t,\hat{\psi}\right)\leq a(t_0)^6E\left(t_0,\hat{\psi}\right)h(t)^2\]
which would diverge since $h(t)$ diverges logarithmically approaching $t=0$. Thus, one would need to rescale the energy by some function approaching $0$ toward the Big Bang faster than $a(t)^6$ to obtain any type of energy convergence. This rescaling would then have to be carried over the proof of Lemma \ref{lem:limit-equality}, or more precisely \eqref{eq:limit-equality-terms-listed}, killing both the entire left hand side since we know that term to be bounded by Proposition $\ref{prop:energy-1}$ and also the $\|A\|^2_{L^2\left(\M\right)}$-term on the right hand side used to relate the energies with $A$. Thus, this \todo{Proposition} and with it the entire approach to our \todo{improved} global and pointwise blow-up conditions as in Theorem \ref{thm:Global-Blowup-hyp} fail.\\

\todo{These improved criteria have multiple advantages: The global blow-up conditions only depend on geometric data in so far as they occur within the volume form and the Laplace-Beltrami operator. In particular, no abstract metric dependent constants arising from Sobolev embedding have to be computed, which makes their application to inhomogeneous geometries significantly easier than Theorem \ref{thm:pointwise-blowup-stiff-incl}. Further, Theorem \ref{thm:pointwise-blowup} state similar initial data requirements to Corollary \ref{cor:pointwise-blowup-stiff-incl}, only at lower regularity since, essentially, the highest order blow-up rate is \enquote{strong enough} to force energy convergence. On the other hand, only Theorem \ref{thm:pointwise-blowup-stiff-incl} and Corollary \ref{cor:pointwise-blowup-stiff-incl} allow to treat one of the most physically relevant cases, the stiff fluid. Altogether, this indicates that the blow-up of matter to FLRW solutions to the Einstein scalar field equations with $\kappa=-1$ is generic, but that one must likely take some care with the behaviour of suitable energies toward the Big Bang hypersurface and impose somewhat harsh open initial data conditions to see this genericity.}
\end{remark}


\addtocounter{page}{1}
\addcontentsline{toc}{chapter}{Bibliography}
\addtocounter{page}{-1}

\bibliographystyle{habbrv}

\bibliography{manuscript}

\end{document}